\documentclass[preprint,12pt]{elsarticle}
\usepackage{amsmath,amsthm,amsfonts}
\usepackage{amssymb}
\usepackage{graphicx}
\usepackage{algorithm}
\usepackage{algpseudocode}
\usepackage{algorithmicx}
\usepackage[table]{xcolor}
\usepackage{hyperref}

\newtheorem{theorem}{Theorem}
\newtheorem{lemma}[theorem]{Lemma}
\newtheorem{proposition}[theorem]{Proposition}
\newtheorem{example}[theorem]{Example}

\newtheorem{problem}[theorem]{Problem}

\begin{document}
\begin{frontmatter}
\title{Heffter arrays over partial loops}

\author{Ra\'{u}l M.\ Falc\'{o}n}
\address{School of Building Engineering University of Seville, Spain}
\ead{rafalgan@us.es}

\author{Lorenzo Mella}
\address{Dip. di Scienze Fisiche, Informatiche, Matematiche, Universit\`a degli Studi di Modena e Reggio Emilia, Via Campi 213/A, I-41125 Modena, Italy}
\ead{lorenzo.mella@unipr.it}
\date{}

\begin{abstract}
A Heffter array over an additive group $G$ is any partially filled array $A$ satisfying that: (1) each one of its rows and columns sum to zero in $G$, and (2) if $i\in G\setminus\{0\}$, then either $i$ or $-i$ appears exactly once in $A$. In this paper, this notion is naturally generalized to that of $\mathcal{B}$-Heffter array over a partial loop, where $\mathcal{B}$ is a set of block-sum polynomials over an affine $1$-design on the set of entries in $A$.
\end{abstract}

\begin{keyword}
Heffter array \sep Latin square \sep quasigroup.
\MSC[2010] 05B15 \sep 20N05
\end{keyword}

\end{frontmatter}
\makeatletter
\def\ps@pprintTitle{%
  \let\@oddhead\@empty
  \let\@evenhead\@empty
  \let\@oddfoot\@empty
  \let\@evenfoot\@oddfoot
}
\makeatother

\begin{abstract}
A Heffter array over an additive group $G$ is any partially filled array $A$ satisfying that: (1) each one of its rows and columns sum to zero in $G$, and (2) if $i\in G\setminus\{0\}$, then either $i$ or $-i$ appears exactly once in $A$. In this paper, this notion is naturally generalized to that of $\mathcal{P}$- and $\mathcal{D}$-Heffter arrays over partial loops, where $\mathcal{P}$ is a set of all one polynomials without constant terms, and $\mathcal{D}$ is an affine $1$-design, both of them related to the entries in $A$.
\end{abstract}

\section{Introduction}

In 2015, Archdeacon \cite{Archdeacon2015}  introduced the notion of \textit{Heffter system} that is defined as follows: given an  integer $m$, a \textit{half-set} $L$ is an  $m$-subset of an abelian group $(G,+)$ of order $2m+1$, such that it contains exactly one of each pair $\{x,-x\}$ and $0 \not\in L$. Then, a \textit{Heffter system} $D(m;k)$ is a partition $\mathcal{B}$ of $L$ into zero-sum $k$-subsets. 
Two Heffter systems $\mathcal{B}_1$ and $\mathcal{B}_2$, defined on the same half-set are said to be \textit{orthogonal} if for every $B_1 \in \mathcal{B}_1$ and $B_2 \in \mathcal{B}_2$ it holds $|B_1 \cap B_2| \leq 1$. (In the original definition, these Heffter systems are called \textit{sub-orthogonal}). 

Archdeacon then arranged two orthogonal Heffter systems in an array, that he called \textit{Heffter array}. In particular, he gave an equivalent definition for Heffter arrays, that is currently the most commonly used one. Given the additive group $\left(\mathbb{Z}_{2nk+1},+\right)$, a \textit{Heffter array $H(m,n;h,k)$} is an $m\times n$ partially filled array  with entries in $\mathbb{Z}_{2nk+1}$ such that the following conditions hold.
\begin{itemize}
    \item[(1)] Each row contains $h$ filled cells and each column contains $k$ filled cells.
    \item[(2)] For each $i\in\mathbb{Z}_{2nk+1}\setminus \{0\}$, either $i$ or $-i$ appears exactly once in the array.
    \item[(3)] The elements in every row and column sum to 0 in $\left(\mathbb{Z}_{2nk+1},+\right)$.
\end{itemize}
If $m = n$, then necessarily $h = k$, and the array is denoted by $H(n; k)$. If no empty cells exist, namely if $h=n$ and $k=m$, then it is  denoted by $H(m,n)$. The most complete existence results on Heffter arrays are the following.
\begin{theorem}\label{thm:Heffter}{\rm \cite{ADDY,CDDY,DW}}
There exists an $H(n; k)$ if and only if $3 \leq k \leq  n$.
\end{theorem}
\begin{theorem}{\rm \cite{ABD}}
There exists an $H(m,n)$ if and only if $m,n\geq 3$.
\end{theorem}
Results on the rectangular case with empty cells, variants and generalizations of the original concept, and related topics can be found in \cite{PD}. 
Of particular interest for the aims of this paper, we remark the generalization recently suggested by Buratti \cite{Buratti2023} of Heffter arrays over non-abelian groups, that further extends the definition of an Heffter array over an abelian group \cite{CMPPBiembed}. The lack of commutativity implies the necessity of ordering those symbols whose summation must be computed. A similar observation has also been indicated in \cite{MT} for generalized Heffter arrays. Our study goes one step further by relaxing not only commutativity but also associativity. This extension to non-associative algebraic structures requires a much deeper formalism for computing the summation of symbols. Even more, we also explore the possibility of dealing with partial binary operations over which the corresponding summation is done. In short, we introduce in this paper the notion of Heffter array over partial loops. A main goal here is showing the existence of these new objects beyond the case of dealing with abelian groups.

Furthermore, the possibility of having more than two mutually orthogonal Heffter systems over the same half-set has recently been explored in \cite{BP}, and further studied in \cite{BJMP,BP2}. In particular, a $(v,k;r)$ \textit{Heffter space} is a set of $r$ mutually orthogonal Heffter systems $\{P_1,\dotsc,$ $ P_r\}$ over the same half-set, where each $P_i$ is a Heffter system $D(v,k)$. In this paper
we study the natural extension of this concept to Heffter spaces over partial loops.

The paper is organized as follows. Section \ref{sec:preliminaries} deals with some preliminary concepts and results that are used throughout the paper. In Section \ref{sec:BlockSum}, we make use of non-associative all one polynomials to formalize the summation over partial loops. A notion of compatibility is then introduced to put emphasis on those polynomials over which summations can be computed from a given partial loop. As an illustrative example, we construct in Section \ref{sec:Example} a family of polynomials describing zero sums from a non-associative loop of order $kp$, where $p$ is an odd prime and $k\neq 2$ is a positive integer. Our generalization of Heffter arrays over partial loops are then introduced in Section \ref{sec:Heffter_loop}. Particularly, we study under which conditions one can ensure the existence of a partial loop over which a given partially filled array constitutes a Heffter array. Some constructive examples are described in this regard, so that we can show the potential of our approach.

\section{Preliminaries} \label{sec:preliminaries}

First, we introduce some notation that we use throughout the paper. For each set $X$, with $0\not\in X$, we define the set $\widetilde{X}:=\{x,-x\colon\, x\in X\}\cup\{0\}$, and we denote by $\mathrm{Sym}(X)$ the symmetric group on $X$. In addition, we denote by $\overline{Sym}(\widetilde{X})$ the set of permutations $\pi\in\mathrm{Sym}(\widetilde{X})$ such that $\pi(0)=0$ and $\pi(-x)=x$ for all $x\in X$. Furthermore, for each positive integer $n$, we consider the cyclic group $(\mathbb{Z}_n,+)$ and the set $[n]:=\{1,\ldots,n\}$. If $n$ is prime, then we consider the field $(\mathbb{Z}_n,+,\cdot)$; and we write $x^{-1}$ to denote the multiplicative inverse of an element $x$ in the group $(\mathbb{Z}_n\setminus\{0\},\cdot)$, and by $\frac{a}{x}$ we mean $ax^{-1}$. Finally, we consider throughout the paper an ordered set of symbols $S:=\{s_1,\ldots,s_n\}$.

Now, we remind some preliminary concepts on affine $1$-designs, partially filled arrays and partial loops that we use in this paper.

\subsection{Affine designs}

A {\em $1$-$(n,k,\lambda)$ design} on the set $S$ is any set of $k$-subsets (called {\em blocks}) of $S$ such that each element (called {\em point}) in $S$ appears in exactly $\lambda$ blocks. It is {\em affine} if (1) blocks can be partitioned into {\em parallel classes}, each of which partitions in turn the set $S$; and (2) there exists a constant $\mu$ such that any two non-parallel blocks intersect in exactly $\mu$ points.

\subsection{Partially filled arrays}

A {\em partially filled array} of order $m$ over $S$ is any $m\times m$ array $A=(A[i,j])$ with rows and columns indexed in $\mathbb{Z}_m$ and entries in the set $S\cup\{\cdot\}$. If $A[i,j]=\cdot$, then the triple $(i,j,\cdot)$ is said to be an {\em empty entry} in $A$. Hence, the array $A$ is uniquely represented by its set of non-empty entries
\[\mathrm{Ent}(A):=\left\{(i,j,A[i,j])\in \mathbb{Z}_m\times \mathbb{Z}_m\times S\right\}.\]
If $\mathrm{Ent}(A)\supseteq\mathrm{Ent}(A')$ for some partially filled array $A'$ of order $m'\leq m$, then it is said that $A$ {\em contains} $A'$. The cardinality of any subset $T\subseteq \mathrm{Ent}(A)$ constitutes its {\em weight}. In addition, we define the set of symbols of $T$ as
\[\mathrm{Symb}(T):=\left\{s\in S\colon\, \exists\, i,j\in \mathbb{Z}_m  \text{ such that } \left(i,j,s\right)\in T \right\}.\]
By abuse of notation, we define the weight of $A$ as that of $\mathrm{Ent}(A)$, and we also define the set $\mathrm{Symb}(A):=\mathrm{Symb}(\mathrm{Ent}(A))$.

\vspace{0.1cm}

Two partially filled arrays $A_1=(A_1[i,j])$ and $A_2=(A_2[i,j])$ of order $m$ over $S$ are {\em isomorphic} if there exists a permutation $\pi\in\mathrm{Sym}(\mathbb{Z}_m )$ satisfying that
\[A_2[\pi(i),\, \pi(j)]=\pi(A_1[i,j])\]
for all $i,j\leq m$ such that $A_1[i,j]\in S$.

\vspace{0.1cm}

Let $\mathcal{A}(m,S)$ denote the set of partially filled arrays of order $m$ over the set $S$ so that each symbol in $S$ appears exactly once within the whole array. (If the array does not contain empty entries, then it has recently been termed a {\em pattern} \cite{Garbe2023}.) For each array $A\in\mathcal{A}(m,S)$, we denote by $\mathrm{Aff}_{\lambda}(A)$ the set of affine 1-$(m^2,m,\lambda)$ designs over the set of entries of $A$, so that every block has weight at least three, and any two non-parallel blocks intersect in exactly one entry. Throughout this paper, we deal with affine $1$-designs formed by at least two of the following four partitions of the set of entries of $A$.

\begin{itemize}
\item[$\mathrm{(1)}$] The {\em row-partition} $\mathrm{Part}_{\mathrm{row}}(A):=\left\{\left\{(i,j,A[i,j])\colon\,j\in \mathbb{Z}_m \right\}\colon\,i\in \mathbb{Z}_m \right\}$.
\item[$\mathrm{(2)}$] The {\em column-partition} $\mathrm{Part}_{\mathrm{col}}(A):=\left\{\left\{(i,j,A[i,j])\colon\,i\in \mathbb{Z}_m \right\}\colon\,j\in \mathbb{Z}_m \right\}$.
\item[$\mathrm{(3)}$] The {\em diagonal-partition} $\mathrm{Part}_{\mathrm{diag}}(A):=\{\{(i,i+j,A[i,i+j])\colon\,i\in \mathbb{Z}_m \}\colon\,j\in \mathbb{Z}_m \}$.
\item[$\mathrm{(4)}$] The {\em anti-diagonal-partition} $\mathrm{Part}_{\mathrm{Adiag}}(A):=\{\{(i,-i-j-1,A[i,-i-j-1])\colon\,i\in \mathbb{Z}_m \}\colon\,j\in \mathbb{Z}_m \}$.
\end{itemize}

\vspace{0.2cm}

We illustrate these partitions with a pair of examples.

\begin{example}\label{example_Affine} Let us consider the array
\[A\equiv\begin{array}{|c|c|c|}\hline
1 & 2 & 3\\ \hline
6 & 4 & 5\\ \hline
8 & 9 & 7 \\ \hline
\end{array}\in\mathcal{A}(3,[9]).\]
We partition its set of entries in the following four ways. (Blocks within each partition are distinguished by highlighting the corresponding cells in different colors.)
\[\begin{array}{ccccccc} \begin{array}{|c|c|c|}\hline
\cellcolor{red!100}{\color{white} 1}& \cellcolor{red!100}{\color{white} 2} & \cellcolor{red!100}{\color{white} 3}\\ \hline
\cellcolor{red!50}{\color{white} 6} & \cellcolor{red!50}{\color{white} 4} & \cellcolor{red!50}{\color{white} 5}\\ \hline
\cellcolor{red!25}{\color{white} 8} & \cellcolor{red!25}{\color{white} 9} & \cellcolor{red!25}{\color{white} 7}\\ \hline
\end{array} &  &
\begin{array}{|c|c|c|}\hline
\cellcolor{blue!100}{\color{white} 1}& \cellcolor{blue!50}{\color{white} 2} & \cellcolor{blue!25}{\color{white} 3}\\ \hline
\cellcolor{blue!100}{\color{white} 6} & \cellcolor{blue!50}{\color{white} 4} & \cellcolor{blue!25}{\color{white} 5}\\ \hline
\cellcolor{blue!100}{\color{white} 8} & \cellcolor{blue!50}{\color{white} 9} & \cellcolor{blue!25}{\color{white} 7}\\ \hline
\end{array} & & \begin{array}{|c|c|c|}\hline
\cellcolor{orange!100}{\color{white} 1}& \cellcolor{orange!50}{\color{white} 2} & \cellcolor{orange!25}{\color{white} 3}\\ \hline
\cellcolor{orange!25}{\color{white} 6} & \cellcolor{orange!100}{\color{white} 4} & \cellcolor{orange!50}{\color{white} 5}\\ \hline
\cellcolor{orange!50}{\color{white} 8} & \cellcolor{orange!25}{\color{white} 9} & \cellcolor{orange!100}{\color{white} 7}\\ \hline
\end{array} & & \begin{array}{|c|c|c|}\hline
\cellcolor{black!100}{\color{white} 1}& \cellcolor{black!50}{\color{white} 2} & \cellcolor{black!25}{\color{white} 3}\\ \hline
\cellcolor{black!50}{\color{white} 6} & \cellcolor{black!25}{\color{white} 4} & \cellcolor{black!100}{\color{white} 5}\\ \hline
\cellcolor{black!25}{\color{white} 8} & \cellcolor{black!100}{\color{white} 9} & \cellcolor{black!50}{\color{white} 7}\\ \hline
\end{array}\\
\mathrm{Part}_{\mathrm{row}}(A) & & \mathrm{Part}_{\mathrm{col}}(A) & & \mathrm{Part}_{\mathrm{diag}}(A) & &
\mathrm{Part}_{\mathrm{Adiag}}(A)
\end{array}
\]
Thus, $\mathcal{D}_A:=\{\mathrm{Part}_{\mathrm{row}}(A),\,\mathrm{Part}_{\mathrm{col}}(A),$ $ \mathrm{Part}_{\mathrm{diag}}(A),$ $\mathrm{Part}_{\mathrm{Adiag}}(A)\}$ is an affine design over the entries of $A$, namely $\mathcal{D}_A \in \mathrm{Aff}_4(A)$.

\hfill $\lhd$
\end{example}

\vspace{0.2cm}

\begin{example}\label{example_Affine_B}Let us consider the array
\[B\equiv{\footnotesize \begin{array}{|c|c|c|c|}\hline
1 & 2 & 3 & \cdot \\ \hline
\cdot & 4 & 5 & 6\\ \hline
9 & \cdot & 7 & 8 \\ \hline
11 & 12 & \cdot & 10 \\ \hline
\end{array}}\in\mathcal{A}(4,[12]).\]
We partition its set of entries in the following four ways.
\[\begin{array}{ccccccccc} {\footnotesize \begin{array}{|c|c|c|c|}\hline
\cellcolor{red!100}{\color{white} 1} & \cellcolor{red!100}{\color{white} 2} & \cellcolor{red!100}{\color{white} 3} & \cellcolor{red!100}{\color{white} \cdot} \\ \hline
\cellcolor{red!50}{\color{white} \cdot} & \cellcolor{red!50}{\color{white} 4} & \cellcolor{red!50}{\color{white} 5} & \cellcolor{red!50}{\color{white} 6}\\ \hline
\cellcolor{red!30}{\color{white} 9} & \cellcolor{red!30}{\color{white} \cdot} & \cellcolor{red!30}{\color{white} 7} & \cellcolor{red!30}{\color{white} 8} \\ \hline
\cellcolor{red!15}{\color{white} 11} & \cellcolor{red!15}{\color{white} 12} & \cellcolor{red!15}{\color{white} \cdot} & \cellcolor{red!15}{\color{white} 10} \\ \hline
\end{array}} &
{\footnotesize \begin{array}{|c|c|c|c|}\hline
\cellcolor{blue!100}{\color{white} 1} & \cellcolor{blue!50}{\color{white} 2} & \cellcolor{blue!30}{\color{white} 3} & \cellcolor{blue!15}{\color{white} \cdot}  \\ \hline
\cellcolor{blue!100}{\color{white} \cdot} & \cellcolor{blue!50}{\color{white} 4} & \cellcolor{blue!30}{\color{white} 5}  & \cellcolor{blue!15}{\color{white} 6} \\ \hline
\cellcolor{blue!100}{\color{white} 9} & \cellcolor{blue!50}{\color{white} \cdot} & \cellcolor{blue!30}{\color{white} 7}  & \cellcolor{blue!15}{\color{white} 8}  \\ \hline
\cellcolor{blue!100}{\color{white} 11} & \cellcolor{blue!50}{\color{white} 12} & \cellcolor{blue!30}{\color{white} \cdot}  & \cellcolor{blue!15}{\color{white} 10}  \\ \hline
\end{array}} & {\footnotesize \begin{array}{|c|c|c|c|}\hline
\cellcolor{orange!100}{\color{white} 1} & \cellcolor{orange!50}{\color{white} 2}  & \cellcolor{orange!30}{\color{white} 3}  & \cellcolor{orange!15}{\color{white} \cdot}  \\ \hline
\cellcolor{orange!15}{\color{white} \cdot}  & \cellcolor{orange!100}{\color{white} 4}  & \cellcolor{orange!50}{\color{white} 5}  & \cellcolor{orange!30}{\color{white} 6} \\ \hline
\cellcolor{orange!30}{\color{white} 9}  & \cellcolor{orange!15}{\color{white} \cdot} & \cellcolor{orange!100}{\color{white} 7}  & \cellcolor{orange!50}{\color{white} 8}  \\ \hline
\cellcolor{orange!50}{\color{white} 11}  & \cellcolor{orange!30}{\color{white} 12}  & \cellcolor{orange!15}{\color{white} \cdot}  & \cellcolor{orange!100}{\color{white} 10} \\ \hline
\end{array}} & {\footnotesize \begin{array}{|c|c|c|c|}\hline
\cellcolor{black!100}{\color{white} 1} & \cellcolor{black!50}{\color{white} 2} & \cellcolor{black!30}{\color{white} 3} & \cellcolor{black!15}{\color{white} \cdot}  \\ \hline
\cellcolor{black!50}{\color{white} \cdot} & \cellcolor{black!30}{\color{white} 4}  & \cellcolor{black!15}{\color{white} 5} & \cellcolor{black!100}{\color{white} 6}\\ \hline
\cellcolor{black!30}{\color{white} 9}  & \cellcolor{black!15}{\color{white} \cdot} & \cellcolor{black!100}{\color{white} 7} & \cellcolor{black!50}{\color{white} 8} \\ \hline
\cellcolor{black!15}{\color{white} 11} & \cellcolor{black!100}{\color{white} 12} & \cellcolor{black!50}{\color{white} \cdot} & \cellcolor{black!30}{\color{white} 10}  \\ \hline
\end{array} }\\
\mathrm{Part}_{\mathrm{row}}(B) & \mathrm{Part}_{\mathrm{col}}(B) & \mathrm{Part}_{\mathrm{diag}}(B) & \mathrm{Part}_{\mathrm{Adiag}}(B)
\end{array}
\]
In particular, $\mathcal{D}_B:=\left\{\mathrm{Part}_{\mathrm{row}}(B),\,\mathrm{Part}_{\mathrm{col}}(B),\,\mathrm{Part}_{\mathrm{diag}}(B)\right\}\in\mathrm{Aff}_3(B)$. Note that any two non parallel classes in $\mathrm{Part}_{\mathrm{diag}}(B)$ and $\mathrm{Part}_{\mathrm{Adiag}}(B)$ intersect in either two or zero points. Hence, the set $\{\mathrm{Part}_{\mathrm{row}}(B),$ $\mathrm{Part}_{\mathrm{col}}(B),$ $\mathrm{Part}_{\mathrm{diag}}(B),$ $\mathrm{Part}_{\mathrm{Adiag}}(B)\}$ is not an affine design. \hfill $\lhd$
\end{example}

\subsection{Partial loops}

A {\em partial loop} over the set $\widetilde{S}$ is any pair $(\widetilde{S},+)$ where $+$ is a partial binary operation over $\widetilde{S}$ such that the following statements hold.
\begin{enumerate}
    \item Each equation $s+x=t$ and $x+s=t$ has at most one solution $x\in \widetilde{S}$, for all $s,t\in \widetilde{S}$.
    \item $0+s = s+0=s$, for all $s\in \widetilde{S}$. (From the first statement, $0$ is the unique element in $\widetilde{S}$ satisfying this condition. Hence, it constitutes the {\em unity element} of the partial loop.)
    \item $s+(-s)=(-s)+s=0$, for all $s\in S$. (Again from the first statement, $-s$ is the unique element in $\widetilde{S}$ satisfying this condition. Hence, it constitutes the {\em inverse element} of $s$.)
\end{enumerate}
It is a {\em loop} if there exists exactly one solution for each equation in the first statement, whereas it is a {\em (partial) group} if the operation $+$ is associative. Its Cayley table is a partially filled array taking the value $\cdot$ in a cell $(s,t)$ whenever the operation  $s+t$ is not defined for some  $s,t\in \widetilde{S}$. By abuse of notation, we represent this fact as $s+t=\cdot$. Even more, we also define $\cdot + s = s+ \cdot = \cdot$, for all $s\in \widetilde{S}\cup\{\cdot\}$.

The Cayley table of a (partial) loop is also termed a (partial) loop. It is a particular type of {\em (partial) Latin square} of order $n$ over $\widetilde{S}$. That is, each symbol in $\widetilde{S}$ appears at most once per row, and at most once per column. (Exactly once if $(\widetilde{S},+)$ is a loop, in whose case, its Cayley table is a Latin square.) From here on, we denote by $\mathcal{L}(\widetilde{S})$ the set of Cayley tables of partial loops over $\widetilde{S}$. If there is risk of confusion, then we denote the partial binary operation of a partial loop $L\in\mathcal{L}(\widetilde{S})$ by $+_L$.

Every permutation $\pi\in\mathrm{Sym}(\widetilde{S})$ constitutes an {\em isomorphism} from a partial loop $L\in\mathcal{L}(\widetilde{S})$  to the partial loop $L^\pi\in\mathcal{L}(\widetilde{S})$ that is defined so that
\begin{equation}
    \label{eq:entries}\mathrm{Ent}(L^\pi):=\left\{(\pi(i),\,\pi(j),\,\pi(L[i,j]))\colon\, (i,j,L[i,j])\in\mathrm{Ent}(L)\right\}.
\end{equation}
If $L^\pi=L$, then $\pi$ is an {\em automorphism} of $L$. The set formed by all the automorphisms of $L$ has group structure. It is called the {\em automorphism group} of $L$, and is denoted by $\mathrm{Aut}(L)$.

\section{Sum polynomials}\label{sec:BlockSum}

Unlike abelian groups, over which Heffter arrays are usually defined, the general lack of commutativity and associativity in partial loops forces us to be careful with the way in which the summation of symbols in a Heffter array over a partial loop is defined.  Non-associative polynomials can be used to formalize this. More specifically, for each ordered subset of symbols $T:=\{t_1,\ldots,t_m\}\subseteq S$, we define the set of variables $X_T:=\{x_{t_1},\ldots,x_{t_m}\}$. A {\em $T$-sum polynomial} is any non-associative all one polynomial
\begin{equation}   \label{eq_Bsum}
\left(\ldots\left(x_{\pi(t_1)}+ \ldots+\left(x_{\pi(t_i)}+x_{\pi(t_{i+1})}\right)\ldots+x_{\pi(t_m)}\right)\ldots\right)
\end{equation}
without constant term, where $\pi\in\mathrm{Sym}(T)$, the parentheses arrangement is well-balanced, and each variable in $X_T$ appears exactly once in the polynomial. Particularly, we call {\em natural} to the $T$-sum polynomial
\begin{equation}   \label{eq_sum_natural}
\nu_{T}:=\left(\ldots(((x_{t_1}+x_{t_2})+x_{t_3})+x_{t_4})+\ldots\right)+x_{t_m}.
\end{equation}
More generally, we term {\em sum polynomial} over $S$ to any $T$-sum polynomial, with $T\subseteq S$. Thus, for instance, $(x_2+(x_4+x_3)+x_1)$, $(x_1+x_2)+(x_3+x_4)$ and $\nu_{[4]}=(((x_1+x_2)+x_3)+x_4)$ are three $[4]$-sum polynomials. The third one is natural.

Every $T$-sum polynomial over $S$ describes a way to sum all the elements in $T$ from any given partial loop over $\widetilde{S}$. More specifically, if $f$ is the sum polynomial described in (\ref{eq_Bsum}), then we define the map
{\small\begin{equation}\label{eq_sigma}\begin{array}{cccc}
\sigma_f: & \mathcal{L}(\widetilde{S}) & \to & \widetilde{S}\cup\{\cdot\}\\
& L & \mapsto & \left(\ldots\left(\pi(t_1)+_L \ldots+_L\left(\pi(t_i)+_L \pi(t_{i+1})\right)\ldots+_L \pi(t_m)\right)\ldots\right)
\end{array}
\end{equation}}

\noindent where each variable $x_{\pi(t_i)}$ in $f$ has been replaced by the corresponding symbol $\pi(t_i)\in T$, and where every sum is computed according to $L$. If the latter is an abelian (partial) group, then this arrangement is irrelevant, and $\sigma_f(L)$ only depends on the symbols indexing the variables in $f$. In this last regard, for each set $P$ of sum polynomials over $S$, we define the {\em support} of $P$ as the subset  $\mathrm{Support}(P)\subseteq S$ that is formed by all the symbols indexing any variable in any sum polynomial in $P$. Then, we say that a partial loop $L\in\mathcal{L}(\widetilde{S})$ is {\em compatible} over $P$, or {\em $P$-compatible}, if the following two statements hold.
\begin{enumerate}
    \item $\sigma_f(L)\in\widetilde{S}$ for all $f\in P$; and

    \item $\sigma_{f}(L)=\sigma_{g}(L)$ for all $f,\,g\in P$ such that $\mathrm{Support}\left(\left\{f\right\}\right)=\mathrm{Support}(\{g\})$.
\end{enumerate}

Of course, abelian groups are compatible over any set of sum polynomials. The next lemma holds readily from the previous definition.

\begin{lemma}\label{lemma_F} Let $P$ and $Q$ be two sets of sum polynomials over $S$, and let $L\in\mathcal{L}(\widetilde{S})$ be $P$-compatible. Then, the following statements hold.
\begin{enumerate}
    \item If $Q\subseteq P$, then $L$ is also $Q$-compatible.
    \item If $L$ is  $Q$-compatible and $\mathrm{Support}(P)\cap\mathrm{Support}(Q)=\emptyset$, then $L$ is $(P\cup Q)$-compatible.
    \item Let $f\in Q$ be such that $\mathrm{Support}(\{f\})=\mathrm{Support}(\{g\})$ for some $g\in P$. If $L$ is an abelian partial group, then it is $\{g\}$-compatible.
\end{enumerate}
\end{lemma}

\begin{example}\label{example_BlockSum} Suppose that we are interested in doing a summation of all the elements of the first row in the array $A\in\mathcal{A}(3,[9])$ described in Example \ref{example_Affine}. This summation can be done according to exactly one of the following $12$ distinct $[3]$-sum polynomials.
\[
\begin{array}{lll}
 f_1:=\nu_{[3]}
=(x_1+x_2)+x_3 & f_2:=(x_2+x_1)+x_3  & f_3:=(x_1+x_3)+x_2 \\
  f_4:=(x_3+x_1)+x_2 & f_5:=(x_2+x_3)+x_1 & f_6:=(x_3+x_2)+x_1  \\
     f_7:=x_1+(x_2+x_3)& f_8:=x_1+(x_3+x_2)&f_9:=x_2+(x_1+x_3)  \\
  f_{10}:=x_2+(x_3+x_1)& f_{11}:=x_3+(x_1+x_2)& f_{12}:=x_3+(x_2+x_1)  \\
\end{array}
\]
Each one of them describes a way to sum all the elements in the set $[3]$ from any given partial loop over the set $\widetilde{[9]}$. To illustrate this fact, we consider the following partial loop in $\mathcal{L}(\widetilde{[9]})$.

{\scriptsize \[L\equiv  \begin{array}{|r||c|c|c|c|c|c|c|c|c|c|c|c|c|c|c|c|c|c|c|}\hline
0 & 1 & 2 & 3 & 4 & 5 & 6 & 7 & 8 & 9 & -9 & -8 & -7 & -6 & -5 & -4 & -3 & -2 & -1 \\
\hline \hline
1 & \cdot & -3 & 2 & 3 & \cdot & \cdot & \cdot & \cdot & \cdot & \cdot & \cdot & \cdot & \cdot & \cdot & \cdot & \cdot & \cdot & 0\\ \hline
2 & -1 & 1 & 4 & \cdot & \cdot & \cdot & \cdot & \cdot & \cdot & \cdot & \cdot & \cdot & \cdot & \cdot & \cdot & \cdot & 0 & \cdot\\ \hline
3 & -2 & -4 & \cdot & \cdot & \cdot & \cdot & \cdot & \cdot & \cdot & \cdot & \cdot & \cdot & \cdot & \cdot & \cdot & 0 & \cdot & 1\\ \hline
4 & \cdot & 3 & \cdot & \cdot & \cdot & \cdot & \cdot & \cdot & \cdot & \cdot & \cdot & \cdot & \cdot & \cdot & 0 & \cdot & \cdot & \cdot\\ \hline
5 & \cdot & \cdot & \cdot & \cdot & \cdot & \cdot & \cdot & \cdot & \cdot & \cdot & \cdot & \cdot & \cdot & 0 & \cdot & \cdot & \cdot & \cdot\\ \hline
6 & \cdot & \cdot & \cdot & \cdot & \cdot & \cdot & \cdot & \cdot & \cdot & \cdot & \cdot & \cdot & 0 & \cdot & \cdot & \cdot & \cdot & \cdot\\ \hline
7 & \cdot & \cdot & \cdot & \cdot & \cdot & \cdot & \cdot & \cdot & \cdot & \cdot & \cdot & 0 & \cdot & \cdot & \cdot & \cdot & \cdot & \cdot\\ \hline
8 & \cdot & \cdot & \cdot & \cdot & \cdot & \cdot & \cdot & \cdot & \cdot & \cdot & 0 & \cdot & \cdot & \cdot & \cdot & \cdot & \cdot & \cdot\\ \hline
9 & \cdot & \cdot & \cdot & \cdot & \cdot & \cdot & \cdot & \cdot & \cdot & 0 & \cdot & \cdot & \cdot & \cdot & \cdot & \cdot & \cdot & \cdot\\ \hline
-9 & \cdot & \cdot & \cdot & \cdot & \cdot & \cdot & \cdot & \cdot & 0 & \cdot & \cdot & \cdot & \cdot & \cdot & \cdot & \cdot & \cdot & \cdot \\ \hline
-8 & \cdot & \cdot & \cdot & \cdot & \cdot & \cdot & \cdot & 0 & \cdot & \cdot & \cdot & \cdot & \cdot & \cdot & \cdot & \cdot & \cdot & \cdot \\ \hline
-7 & \cdot & \cdot & \cdot & \cdot & \cdot & \cdot & 0 & \cdot & \cdot & \cdot & \cdot & \cdot & \cdot & \cdot & \cdot & \cdot & \cdot & \cdot \\ \hline
-6 & \cdot & \cdot & \cdot & \cdot & \cdot & 0 & \cdot & \cdot & \cdot & \cdot & \cdot & \cdot & \cdot & \cdot & \cdot & \cdot & \cdot & \cdot \\ \hline
-5 & \cdot & \cdot & \cdot & \cdot & 0 & \cdot & \cdot & \cdot & \cdot & \cdot & \cdot & \cdot & \cdot & \cdot & \cdot & \cdot & \cdot & \cdot \\ \hline
-4 & 2 & \cdot & \cdot & 0 & \cdot & \cdot & \cdot & \cdot & \cdot & \cdot & \cdot & \cdot & \cdot & \cdot & \cdot & \cdot & \cdot & \cdot \\ \hline
-3  & \cdot & \cdot & 0 & \cdot & \cdot & \cdot & \cdot & \cdot & \cdot & \cdot & \cdot & \cdot & \cdot & \cdot & \cdot & \cdot & \cdot & \cdot \\ \hline
-2 & \cdot & 0 & \cdot & \cdot & \cdot & \cdot & \cdot & \cdot & \cdot & \cdot & \cdot & \cdot & \cdot & \cdot & \cdot & \cdot & \cdot & \cdot \\ \hline
-1 & 0 & \cdot & 1 & \cdot & \cdot & \cdot & \cdot & \cdot & \cdot & \cdot & \cdot & \cdot & \cdot & \cdot & \cdot & \cdot & \cdot & \cdot \\ \hline
\end{array}\]}

We have for instance that
\[\sigma_{f_1}(L)=(1+_L 2)+_L 3=L[L[1,2],3]=L[-3,3]=0.\]
In a similar way,
\[\sigma_{f_1}(L)=\sigma_{f_4}(L)=\sigma_{f_{10}}(L)=\sigma_{f_{11}}(L)=0, \hspace{2cm} \sigma_{f_5}= \sigma_{f_8}(L)=\cdot,\]
\[\sigma_{f_2}(L)=\sigma_{f_3}(L)=\sigma_{f_9}(L)=\sigma_{f_{12}}(L)=1,\hspace{0.4cm} \sigma_{f_6}(L)=2 \hspace{0.4cm} \text{ and }  \hspace{0.4cm} \sigma_{f_7}(L)=3.\]
Hence, $L$ is $P$-compatible for any set of sum polynomials $P\in\{\{f_1,\,f_4,\,f_{10},$ $\,f_{11}\},\, \{f_2,\,f_3,\,f_9,\,f_{12}\},\, \{f_6\},\, \{f_7\}\}$. \hfill $\lhd$
\end{example}

\vspace{0.25cm}

The automorphism group of a partial loop plays a relevant role to determine sum polynomials over which the partial loop is compatible. To see it, let $f$ be the sum polynomial described in (\ref{eq_Bsum}). For each permutation $\rho\in \mathrm{Sym}(\widetilde{S})$ such that $\rho(s)\in S$ for all $s\in\mathrm{Support}(\{f\})$, we define the sum polynomial
{\small \begin{equation}   \label{eq_Bsum_pi}
f^\rho:= \left(\ldots\left(x_{\rho(\pi(t_1))}+ \ldots+\left(x_{\rho(\pi(t_i))}+x_{\rho(\pi(t_{i+1}))}\right)\ldots+x_{\rho(\pi(t_m))}\right)\ldots\right)
\end{equation}}

\noindent where the parentheses arrangement coincides with that of (\ref{eq_Bsum}). In addition, for each set $P$ of sum polynomials over $S$, we define the set $P^{\rho}:=\left\{f^{\rho}\colon\,f\in P\right\}$. The following result shows how this permutation preserves the compatibility of the partial loop under consideration.

\begin{lemma}\label{lemma_isopolynomial}  Let $P$ be a set of sum polynomials over $S$. If $L\in \mathcal{L}(\widetilde{S})$ and $\pi\in \mathrm{Sym}(\widetilde{S})$ satisfies that $\mathrm{Support}(P^\pi)\subseteq S$, then the following statements hold.
\begin{enumerate}
    \item $\sigma_{f^\pi}(L^{\pi})=\pi(\sigma_f(L))$ for all $f\in P$.
\item $L$ is $P$-compatible if and only if $L^{\pi}$ is $P^\pi$-compatible.
\end{enumerate}
\end{lemma}

\begin{proof} The first statement holds readily from (\ref{eq:entries}),  (\ref{eq_sigma}) and (\ref{eq_Bsum_pi}). Since $L^\pi\in\mathcal{L}(\widetilde{S})$. Then, the second one follows straightforwardly.
\end{proof}

\vspace{0.25cm}

Now, for each set $P$ of sum polynomials over $S$ and each partial loop $L\in\mathcal{L}(\widetilde{S})$, we define the subgroup
\[\mathrm{Aut}_P(L):=\left\{\pi\in\mathrm{Aut}(L)\colon\,\mathrm{Support}(P^\pi)\subseteq S\right\}\leq \mathrm{Aut}(L).\]
The following result holds readily from the second statements of Lemmas \ref{lemma_F} and \ref{lemma_isopolynomial}.

\begin{proposition}\label{proposition_isopolynomial}  If the partial Latin loop $L$ is $P$-compatible, then it is also $\bigcup_{\pi\in\mathrm{Aut}_P(L)}P^\pi$-compatible.
\end{proposition}

\section{A family of sum polynomials describing zero sums from non-associative loops}\label{sec:Example}

In order to illustrate the concepts and results described in the previous section, we construct here a family of sum polynomials describing zero sums over non-associative loops of order $kp$, where $p$ is an odd prime and $k\neq 2$ is a positive integer. To this end, we distinguish two constructions: a basic one focused on the case $k=1$, and the general construction.

\subsection{The basic construction}

Let $r$ be a primitive root of $p$. Then, we define the $p\times p$ array $H_{p,r}$ over $\mathbb{Z}_p$ so that
\begin{equation}
    \label{eq_EqHpr} \mathrm{Ent}(H_{p,r}):=\left\{\left(i,j, j+\frac{i-j}r\right)\colon\, i,j\in \mathbb{Z}_p\right\}.
\end{equation}
It is {\em diagonally cyclic}. That is, the permutation $i\mapsto a\cdot i+b$ in $\mathrm{Sym}(\mathbb{Z}_p)$ is an automorphism of $H_{p,r}$ for all $a,b\in\mathbb{Z}_p$, with $a\neq 0$. In addition, since $r$ is a primitive root of $p$, both sets of symbols in the first row and first column of $H_{p,r}$ coincide with $\mathbb{Z}_p$. Hence, the array $H_{p,r}$ is indeed a Latin square. It is symmetric if and only if $r=2$. Due to the cyclic character of the array, this fact can be proved by studying the coincidence of symbols in symmetric cells within its first row and first column. In particular, if $i\neq 0$, then
\[i-\frac ir = \frac ir \Leftrightarrow i = \frac {2i}r \Leftrightarrow r=2.\]

\begin{example}\label{example_0} We have that
\[H_{7,3}\equiv{\footnotesize \begin{array}{|c|c|c|c|c|c|c|}\hline
         \cellcolor{red!100}{\color{white} 0} &\cellcolor{blue!100}{\color{white} 3} &\cellcolor{green!100}{\color{white} 6}&\cellcolor{orange!100}{\color{white} 2}&\cellcolor{pink!100}{\color{white} 5}&\cellcolor{yellow!100}{\color{white} 1}&\cellcolor{purple!100}{\color{white} 4}  \\ \hline
         \cellcolor{purple!100}{\color{white} 5}&\cellcolor{red!100}{\color{white} 1}&\cellcolor{blue!100}{\color{white} 4}&\cellcolor{green!100}{\color{white} 0}&\cellcolor{orange!100}{\color{white} 3}&\cellcolor{pink!100}{\color{white} 6}&\cellcolor{yellow!100}{\color{white} 2}  \\ \hline
         \cellcolor{yellow!100}{\color{white} 3}&\cellcolor{purple!100}{\color{white} 6}&\cellcolor{red!100}{\color{white} 2} &\cellcolor{blue!100}{\color{white} 5}&\cellcolor{green!100}{\color{white} 1}&\cellcolor{orange!100}{\color{white} 4}&\cellcolor{pink!100}{\color{white} 0} \\ \hline
         \cellcolor{pink!100}{\color{white} 1}&\cellcolor{yellow!100}{\color{white} 4}&\cellcolor{purple!100}{\color{white} 0}&\cellcolor{red!100}{\color{white} 3} &\cellcolor{blue!100}{\color{white} 6}&\cellcolor{green!100}{\color{white} 2}&\cellcolor{orange!100}{\color{white} 5}  \\ \hline
         \cellcolor{orange!100}{\color{white} 6}&\cellcolor{pink!100}{\color{white} 2}&\cellcolor{yellow!100}{\color{white} 5}&\cellcolor{purple!100}{\color{white} 1}&\cellcolor{red!100}{\color{white} 4} &\cellcolor{blue!100}{\color{white} 0}&\cellcolor{green!100}{\color{white} 3} \\ \hline
         \cellcolor{green!100}{\color{white} 4}&\cellcolor{orange!100}{\color{white} 0}&\cellcolor{pink!100}{\color{white} 3}&\cellcolor{yellow!100}{\color{white} 6}&\cellcolor{purple!100}{\color{white} 2}&\cellcolor{red!100} {\color{white} 5}&\cellcolor{blue!100}{\color{white} 1}  \\ \hline
         \cellcolor{blue!100}{\color{white} 2}&\cellcolor{green!100}{\color{white} 5}&\cellcolor{orange!100}{\color{white} 1}&\cellcolor{pink!100}{\color{white} 4}&\cellcolor{yellow!100}{\color{white} 0}&\cellcolor{purple!100}{\color{white} 3}&\cellcolor{red!100}{\color{white} 6}  \\ \hline
    \end{array}}
    \]
where cells are colored to visualize the cyclic character of the array. \hfill $\lhd$
\end{example}

\vspace{0.2cm}

Now, we define the loop  $L_{p,r}\in\mathcal{L}(\widetilde{[p]})$ so that, for each $i,j\in [p]$,
{\small \[\begin{cases}\begin{array}{lll}
L_{p,r}[i,j] & :=-\left(H_{p,r}[i-1,j-1]+1\right),\\
L_{p,r}[i,-j]=L_{p,r}[-i,j] & :=H_{p,r}[i-1,j-1]+1, & \text{ if } i\neq j,\\
L_{p,r}[-i,-j] & :=-\left(H_{p,r}L_{p,r}[i,j][i-1,j-1]+1\right), & \text{ if } i\neq j,\\
L_{p,r}[-i,-i] & :=i.
\end{array}
\end{cases}\]}

\begin{example}\label{example_0a} We have that
\[L_{7,3}\equiv{\scriptsize\begin{array}{|c||c|c|c|c|c|c|c||c|c|c|c|c|c|c|} \hline
0&1&2&3&4&5&6&7&-7&-6&-5&-4&-3&-2&-1 \\ \hline\hline
    1&-1&-4&-7&-3&-6&-2&-5& 5&2&6&3&7&4&0 \\ \hline
    2& -6&-2&-5&-1&-4&-7&-3&3&7&4&1&5&0&6\\ \hline
    3&  -4&-7&-3&-6&-2&-5&-1&1&5&2&6&0&7&4\\ \hline
    4&   -2&-5&-1&-4&-7&-3&-6&6&3&7&0&1&5&2\\ \hline
    5&   -7&-3&-6&-2&-5&-1&-4&4&1&0&2&6&3&7\\ \hline
    6&    -5&-1&-4&-7&-3&-6&-2&2&0&3&7&4&1&5\\ \hline
    7&    -3&-6&-2&-5&-1&-4&-7&0&4&1&5&2&6&3\\ \hline\hline
    -7&3&6&2&5&1&4&0&7&-4&-1&-5&-2&-6&-3\\ \hline
    -6&5&1&4&7&3&0&2&-2&6&-3&-7&-4&-1&-5\\ \hline
    -5&7&3&6&2&0&1&4&-4&-1&5&-2&-6&-3&-7\\ \hline
    -4&2&5&1&0&7&3&6&-6&-3&-7&4&-1&-5&-2\\ \hline
    -3&4&7&0&6&2&5&1&-1&-5&-2&-6&3&-7&-4\\ \hline
    -2&6&0&5&1&4&7&3&-3&-7&-4&-1&-5&2&-6\\ \hline
    -1&0&4&7&3&6&2&5&-5&-2&-6&-3&-7&-4&1\\ \hline
 \end{array}}
\]
\hfill $\lhd$
\end{example}

\vspace{0.2cm}

The properties satisfied by $H_{p,r}$ imply the following statements.
\begin{itemize}
    \item[$\mathrm{(1)}$] The loop $L_{p,r}$ is symmetric if and only if $r = 2$.
    \item[$\mathrm{(2)}$] The loop $L_{p,r}$ is not a group  because, for each $i,j,k\in [p]$, we have that
\[\left(i+_{L_{p,r}} j\right)+_{L_{p,r}} k = \left(k-1+\frac {j+\frac {i-j}r-k}r\pmod p\right) + 1\]
and
\[i+_{L_{p,r}}\left( j+_{L_{p,r}} k\right) =
\left(k-1+\frac{j-k}r + \frac {i- k-\frac{j-k}r}r\pmod p \right)+ 1.\]
That is, $\left(i+_{L_{p,r}} j\right)+_{L_{p,r}} k = i+_{L_{p,r}}\left( j+_{L_{p,r}} k\right)$ if and only if
\begin{align*}
i-j\equiv ir-kr-j-k \pmod p & \Leftrightarrow i(1-r)\equiv k(1-r) \pmod p\\
& \Leftrightarrow i\equiv k \pmod p.
\end{align*}

\item[$\mathrm{(3)}$] The loop $L_{p,r}$ is  $\{\nu_{[p]}\}$-compatible, where $\nu_{[p]}$ is the natural $[p]$-sum polynomial. More precisely, for each positive integer $\ell<p$, we have that
\[\sigma_{\nu_{[\ell]}}(L_{p,r}) = (-1)^{\ell+1}\cdot \left(\left(\frac {r^{1-\ell}-1}{r-1}+\ell-1\pmod p\right)+1\right)\in\widetilde{[p]}.\]
It can be proven by induction on $\ell$. Thus, $\sigma_{\nu_{[1]}}(L_{p,r}) = 1$. Then, we assume that it holds for some $\ell<p-1$, and we show that it also holds for $\ell +1$.  To this end, note that
\[\left(\frac {r^{1-\ell}-1}{r-1}+\ell-1\pmod p\right)+1=\ell+1\Leftrightarrow r^{1-\ell}=r.\]
From Fermat's Little Theorem, this happens if and only if $\ell=p-1$. Then, by the induction hypothesis,
{\small \begin{align*}
\sigma_{\nu_{[\ell+1]}}(L_{p,r}) & =\sigma_{\nu_{[\ell]}}(L_{p,r}) +_{L_{p,r}} (\ell+1)=\\
&=(-1)^{\ell+1}\cdot\left(\left(\frac {r^{1-\ell}-1}{r-1}+\ell-1\pmod p\right) + 1\right)+_{L_{p,r}} (\ell+1)=\\
& =(-1)^{\ell+2}\cdot \left(H_{p,r}\left[\frac {r^{1-\ell}-1}{r-1}+\ell-1\,(\mathrm{mod}\, p),\,\ell\right]+1\right)=\\
& =(-1)^{\ell+2}\cdot \left(\left(\ell+\frac{\frac {r^{1-\ell}-1}{r-1}+\ell-1 -\ell}r\pmod p\right) +1 \right)=\\
& =(-1)^{\ell+2}\cdot \left(\left(\frac {r^{1-(\ell+1)}-1}{r-1}+\ell\pmod p\right) + 1\right).
\end{align*}}
As a consequence,
\begin{equation}\label{eq:zerosum}
\sigma_{\nu_{[p]}}(L_{p,r})=\sigma_{\nu_{[p-1]}}(L_{p,r}) +_{L_{p,r}} p=-p +_{L_{p,r}} p=0.
\end{equation}
\end{itemize}

\begin{example}\label{example_0b} We have that
\begin{align*}
\sigma_{\nu_{[7]}}(L_{7,3}) & =(((((1+2)+3)+4)+5)+6)+7=\\
& =((((-4+3)+4)+5)+6)+7=\\
& =(((1+4)+5)+6)+7=\\
& =((-3+5)+6)+7=\\
& =(2+6)+7=\\
& =-7+7=\\
&=0.
\end{align*}
\hfill $\lhd$
\end{example}

Now, let $a,b\in\mathbb{Z}_p$, with $a\neq 0$. The automorphism group $\mathrm{Aut}(L_{p,r})$ contains the permutation $\pi_{(a,b)}\in\overline{\mathrm{Sym}}(\widetilde{[p]})$ that is defined so that for each $i\in [p]$,
\begin{equation}\label{eq:perm_natural_ab}
\pi_{(a,b)}(i):=\left(a(i-1)+b \pmod p\right)+1.
\end{equation}

\begin{proposition}\label{proposition_isopolynomial_ab} The non-associative loop $L_{p,r}$ is $\{\nu^{\pi_{(a,b)}}_{[p]}\colon\, a,b\in\mathbb{Z}_p,\, a\neq 0\}$-compatible.  More precisely, every sum polynomial in the set describes a zero sum in $[p]$.
\end{proposition}

\begin{proof} Since $\pi_{(a,b)}\in \mathrm{Aut}(L_{p,r})$ for all $a,b\in\mathbb{Z}_p$, with $a\neq 0$, the first statement follows readily from Proposition \ref{proposition_isopolynomial}. The second statement holds from Lemma \ref{lemma_isopolynomial} and (\ref{eq:zerosum}).
\end{proof}

\vspace{0.25cm}

Thus, for instance,
\begin{align*}
\sigma_{\nu^{\pi_{(0,3)}}_{[7]}}(L_{7,3}) & =(((((4+5)+6)+7)+1)+2)+3=\\
& =((((-7+6)+7)+1)+2)+3=\\
& =(((4+7)+1)+2)+3=\\
& =((-6+1)+2)+3=\\
& =((5+2)+3=\\
& =-3+3=\\
&=0
\end{align*}
and
\begin{align*}
\sigma_{\nu^{\pi_{(2,4)}}_{[7]}}(L_{7,3}) & =(((((6+1)+3)+5)+7)+2)+4=\\
& =(((((-5+3)+5)+7)+2)+4=\\
& =((((6+5)+7)+2)+4=\\
& =(((-3+7)+2)+4=\\
& =((1+2)+4=\\
& =-4+4=\\
&=0.
\end{align*}

\subsection{The general construction}

The previous construction can be generalized to define loops that are compatible over sum polynomials on $kp$ variables, where $k\neq 2$ is a positive integer. In this regard, let us consider an ordered $kp$-set of symbols 
\[S:=\bigcup_{\ell\in [k]}\bigcup_{i\in [p]}\{s_{\ell,i}\}.\] For each $\ell\in [k]$, we define the ordered subset $S_\ell:=\{s_{\ell,i}\colon\, i\in [p]\}$ and the element $s_{\ell,0}:=0$. Then, for each permutation $\pi\in\mathrm{Sym}([k])$, we define the following sum polynomial arisen from natural sum polynomials over $S$.
\[\nu_{S_{\pi(1)},\ldots,S_{\pi(k)}}:=\left(\ldots\left(\left(\nu_{S_{\pi(1)}}+\nu_{S_{\pi(2)}}\right)+\nu_{S_{\pi(3)}}\right)+\ldots\right)+\nu_{S_{\pi(k)}}.\]
Further, for each $2k$-tuple $\alpha:=(a_1,b_1,\ldots,a_k,b_k)\in\mathbb{Z}_p^{2k}$ such that $a_\ell\neq 0$ for all $\ell\in [k]$, we also define the permutation $\pi_\alpha\in\overline{\mathrm{Sym}}(\widetilde{S})$ so that, for each $\ell\in [k]$ and $i\in [p]$,
\begin{equation}\label{eq:perm_natural}
\pi_\alpha(s_{\ell,i}):=s_{\ell, (a_\ell(i-1)+b_\ell\pmod p)+1}.
\end{equation}
That is, the restriction of $\pi_\alpha$ to each $S_\ell$ is a equivalent to (\ref{eq:perm_natural_ab}). Finally, let $I$ be an idempotent Latin square of order $k$ over $[k]$. That is, $I[\ell,\ell]=\ell$ for all $\ell\in\mathbb{Z}_k$. (This always exists because $k\neq 2$.) Then, we define the loop $L_{p,r}^{S,I}\in\mathcal{L}(\widetilde{S})$ such that, for each $\ell_1,\ell_2\in [k]$ and each $i_1,i_2\in [p]$,
\[L_{p,r}^{S,I}\left[\pm s_{\ell_1,\,i_1},\, s_{\ell_2,\,i_2}\right]:=\pm s_{I[\ell_1,\ell_2],\,L_{p,r}[i_1,i_2]}\]
and
\[L_{p,r}^{S,I}\left[\pm s_{\ell_1,\,i_1},\, -s_{\ell_2,\,i_2}\right]:=\mp s_{I[\ell_1,\ell_2],\,L_{p,r}[i_1,i_2]}\]

\vspace{0.2cm}

\begin{example}\label{example_I0} If we consider the idempotent Latin square
\[I\equiv \begin{array}{|c|c|c|} \hline
1 & 3 & 2\\ \hline
3 & 2 & 1\\ \hline
2 & 1 & 3\\ \hline
\end{array}\]
then the non-associative loop $L_{3,3}^{[9],I}$ is defined as follows.
\[{\tiny\begin{array}{|c||c|c|c|c|c|c|c|c|c||c|c|c|c|c|c|c|c|c|} \hline
0&1&2&3&4&5&6&7&8&9&-9&-8&-7&-6&-5&-4&-3&-2&-1 \\ \hline\hline
    1&-1&-3&-2&-7&-9&-8&-4& -6&-5&5&6&4&8&9&7&2&3&0\\ \hline
    2& -3&-2&-1&-9&-8&-7&-6&-5&-4&4&5&6&7&8&9&1&0&3\\ \hline
    3&  -2&-1&-3&-8&-7&-9&-5&-4&-6&6&4&5&9&7&8&0&1&2\\ \hline
    4&-7&-9&-8&-4&-6&-5&-1&-3&-2&2&3&1&5&6&0&8&9&7\\ \hline
    5&-9&-8&-7&-6&-5&-4&-3&-2&-1&1&2&3&4&0&6&7&8&9\\ \hline
    6&-8&-7&-9&-5&-4&-6&-2&-1&-3&3&1&2&0&4&5&9&7&8\\ \hline
    7&-4&-6&-5&-1&-3&-2&-7&-9&-8&8&9&0&2&3&1&5&6&4\\ \hline
    8 &-6&-5&-4&-3&-2&-1&-9&-8&-7&7&0&9&1&2&3&4&5&6\\ \hline
    9 &-5&-4&-6&-2&-1&-3&-8&-7&-9&0&7&8&3&1&2&6&4&5\\ \hline\hline
    -9&5&4&6&2&1&3&8&7&0&9&-7&-8&-3&-1&-2&-6&-4&-5\\ \hline
    -8 &6&5&4&3&2&1&9&0&7&-7&8&-9&-1&-2&-3&-4&-5&-6\\ \hline
    -7&4&6&5&1&3&2&0&9&8&-8&-9&7&-2&-3&-1&-5&-6&-4\\ \hline
    -6&8&7&9&5&4&0&2&1&3&-3&-1&-2&6&-4&-5&-9&-7&-8\\ \hline
    -5&9&8&7&6&0&4&3&2&1&-1&-2&-3&-4&5&-6&-7&-8&-9\\ \hline
    -4&7&9&8&0&6&5&1&3&2&-2&-3&-1&-5&-6&4&-8&-9&-7\\ \hline
    -3&2&1&0&8&7&9&5&4&6&-6&-4&-5&-9&-7&-8&3&-1&-2\\ \hline
    -2&3&0&1&9&8&7&6&5&4&-4&-5&-6&-7&-8&-9&-1&2&-3\\ \hline
    -1&0&3&2&7&9&8&4&6&5&-5&-6&-4&-8&-9&-7&-2&-3&1\\ \hline
 \end{array}}
\]
\hfill $\lhd$
\end{example}

\vspace{0.2cm}

\begin{proposition}\label{proposition_isopolynomial_kp} The loop $L_{p,r}^{S,I}$ is $\{\nu^{\pi_{\alpha}}_{S_{\pi(1)},\ldots,S_{\pi(k)}}\colon\, \pi\in\mathrm{Sym}([k]) \text{ and } \alpha=(a_1,b_1,$ $\ldots,a_k,b_k)\in\mathbb{Z}_p^{2k} \text{ is such that } a_\ell\neq 0 \text{ for all } \ell\in [k]\}$-compatible. More precisely, every sum polynomial in the set describes a zero sum in $S$.
\end{proposition}

\begin{proof} Let $\alpha:=(a_1,b_1,\ldots,a_k,b_k)\in\mathbb{Z}_p^{2k}$ such that $a_\ell\neq 0$ for all $\ell\in [k]$. In addition, for each $\ell\in [k]$, let $\pi_{(a_\ell,b_\ell)}\in\overline{\mathrm{Sym}}(\widetilde{[p]})$ be defined as in (\ref{eq:perm_natural_ab}). Since $I$ is idempotent, Proposition \ref{proposition_isopolynomial_ab} implies that
\[\sigma_{\nu^{\pi_\alpha}_{S_\ell}}\left(L_{p,r}^{S,I}\right)=s_{\ell,\,\sigma_{\nu^{\pi_{(a_\ell,b_\ell)}}_{[p]}}(L_{p,r})}=\sigma_{\ell,0}=0.\]
Then, the result holds readily because $L_{p,r}^{S,I}[0,0]=0$.
\end{proof}

\begin{example}\label{example_I0a} Under the assumptions of Example \ref{example_I0}, we have that 
 \begin{align*}
\sigma_{\nu^{(1,0,2,1,2,2)}_{\{1,2,3\},\{7,8,9\},\{4,5,6\}}}\left(L_{3,2}^{[9],I}\right)& =(((1+2)+3) + ((8+7)+9))+ ((6+5)+4))=\\
& =((-3+3) + (-9+9))+ (-4+4))=\\
& = (0+0)+0=\\
& =0.
\end{align*}  
\hfill $\lhd$
\end{example}

\section{Heffter arrays over partial loops}\label{sec:Heffter_loop}

In the previous sections, we have shown some examples of sum polynomials giving rise to zero sums over non-associative (partial) loops. Based on the existence of this kind of examples, we introduce in this section a natural generalization of the concept of Heffter array so that summations are defined over (partial) loops instead that groups. Moreover, we are interested in other block sums apart from row and column ones, so we also generalize the concept of Heffter system from abelian groups to (partial) loops. In this regard, we consider the following objects: a finite set of symbols $S$, with $0\not\in S$; a pair of positive integers $m$ and $\lambda$; a partially filled array $A\in\mathcal{A}(m,S)$; and an affine $1$-design $\mathcal{D}\in\mathrm{Aff}_{\lambda}(A)$. Then, we call {\em $\mathcal{D}$-sum-polynomial set} to any set $\mathcal{P}$ of sum polynomials over $S$ satisfying the following two conditions.
\begin{enumerate}
    \item If $f\in \mathcal{P}$, then there is a block $B\in\mathcal{D}$ such that $\mathrm{Support}(\{f\})=\mathrm{Symb}(B)$.
    \item If $B\in\mathcal{D}$, then there is at least one $\mathrm{Symb}(B)$-sum polynomial in $\mathcal{P}$.
\end{enumerate}
Let $\mathcal{P}$ be a $\mathcal{D}$-sum-polynomial set. We say that the array $A$ is {\em $\mathcal{P}$-Heffter} over a partial loop $L\in\mathcal{L}(\widetilde{S})$ if $\sigma_f(L)=0$ for all $f\in\mathcal{P}$. If $\mathcal{P}$ contains all the $\mathrm{Symb}(B)$-sum polynomials for all $B\in\mathcal{D}$, then we also say that $A$ is {\em $\mathcal{D}$-Heffter} over $L$. Furthermore, if for every pair of distinct symbols in $S$ there exists precisely one block $B \in \mathcal{D}$ containing them, we say that $A$ is a \emph{Heffter linear space} over $L$, while if $\mathcal{D}=\{\mathrm{Part}_{\mathrm{row}}(A),$ $\mathrm{Part}_{\mathrm{col}}(A)\}$, then we say that $A$ is a {\em Heffter array} over $L$. These notions coincide with those of Heffter linear space and Heffter array whenever $L$ is an abelian group. 

\begin{example}\label{example_Heffter} The array
\[A\equiv \begin{array}{|c|c|c|} \hline
1 & 2 & 3\\ \hline
4 & 5 & 6\\ \hline
7 & 8 & 9\\ \hline
\end{array}\]
is Heffter over the non-associative loop $L_{3,3}^{[9],I}$ defined in Example \ref{example_I0}. So, it is an example of a Heffter linear space over a non-associative loop. \hfill $\lhd$
\end{example}

\vspace{0.2cm}

The following problem arises.

\begin{problem}\label{problem_existence} Let $A\in\mathcal{A}(m,S)$ and $\mathcal{D}\in\mathrm{Aff}_\lambda(A)$. Does there exist a partial loop $L\in\mathcal{L}(\widetilde{S})$ and a $\mathcal{D}$-sum-polynomial set $\mathcal{P}$ such that the array $A$ is $\mathcal{P}$-Heffter over $L$? Does there exist a (partial) loop over which the array $A$ is $\mathcal{D}$-Heffter?
\end{problem}

\vspace{0.25cm}

The next theorem shows how the general construction described in Section \ref{sec:Example} gives a partial affirmative answer to the first question in Problem \ref{problem_existence}.

\begin{theorem} \label{thm:loop_for_affine_design} Let $A\in\mathcal{A}(m,S)$ and $\mathcal{D}\in\mathrm{Aff}_\lambda(A)$. If, for each block $B\in\mathcal{D}$, here exists an odd prime $p_B$ and a positive integer $k_B\neq 2$ such that $|\mathrm{Symb}(B)|=k_Bp_B$, then there is a partial loop $L_{\mathcal{D}}\in\mathcal{L}(\widetilde{S})$ such that the array $A$ is $\{\nu^{\pi_{\alpha}}_{\mathrm{Symb}(B)}\colon\, B\in\mathcal{D} \text{ and } \alpha=(a_1,b_1,$ $\ldots,a_k,b_k)\in\mathbb{Z}_p^{2k} \text{ is such that } a_\ell\neq 0 \text{ for all } \ell\in [k]\}$-Heffter over $L_{\mathcal{D}}$. Moreover, if all the blocks in $\mathcal{D}$ have the same size, then there is a loop satisfying this condition.
\end{theorem}

\begin{proof} For each block $B\in \mathcal{D}$, let $r_B$ be a primitive root of the corresponding odd prime $p_B$, and let $I_B$ be an idempotent Latin square of order $k_B$ over the set $[k_B]$. By following the general construction described in Section \ref{sec:Example}, we define the non-associative loop $L_{p_B,r_B}^{\mathrm{Symb}(B),I_B}\in\mathcal{L}(\widetilde{\mathrm{Symb}(B)})$. 

The required partial loop $L_{\mathcal{D}}\in \mathcal{L}(\widetilde{S})$ is defined so that
\begin{equation}\label{eq_LD}
\mathrm{Ent}(L_\mathcal{D}):=\bigcup_{B\in\mathcal{D}} \mathrm{Ent}(L_B).
\end{equation}
It is well-defined because $\mathcal{D}\in\mathrm{Aff}_\lambda(A)$ and hence, every pair of non-parallel blocks in $\mathcal{D}$ intersect in exactly one point. More precisely, if there exist two elements $a,b\in (\widetilde{\mathrm{Symb}(B_1)}\cap\widetilde{\mathrm{Symb}(B_2))}\setminus\{0\}$ for a pair of distinct blocks $B_1,B_2\in\mathcal{D}$, then it must be $b=\pm a$. Then, $L_{\mathcal{D}}[a,b]\in\{a,-a\}\in (\widetilde{\mathrm{Symb}(B_1)}\cap\widetilde{\mathrm{Symb}(B_2))}\setminus\{0\}$.
\end{proof}

\vspace{0.5cm}

Concerning the second question in Problem \ref{problem_existence}, the following preliminary result holds.

\begin{lemma}\label{lemma_existence} Let $A\in\mathcal{A}(m,S)$ and $\mathcal{D}\in\mathrm{Aff}_\lambda(A)$. In addition, for each block $B\in\mathcal{D}$, let $L_B\in\mathcal{L}(\widetilde{S})$ be the smallest partial loop contained in $L$ from which $\sigma_f(L)$ can be computed for any $\mathrm{Symb}(B)$-sum polynomial $f$, and let $L_\mathcal{D}\in\mathcal{L}(\widetilde{S})$ be described as in (\ref{eq_LD}). Then, all the partial loops $L_B$, and also $L_{\mathcal{D}}$, 
are partial abelian groups.
\end{lemma}

\begin{proof} Let $B\in \mathcal{D}$, and let $T:=\{s_1,s_2,s_3\}\subseteq \mathrm{Symb}(B)$. Since $A$ is $\mathcal{D}$-Heffter over $L_\mathcal{D}$, we have that
\[\left((s_1+s_2)+s_3\right)+\sum_{s\in\mathrm{Symb}(B)\setminus T}  s=0=\left(s_1+(s_2+s_3)\right)+\sum_{s\in\mathrm{Symb}(B)\setminus T}  s.\]
Since $L_B$ is a partial Latin square, then $(s_1+s_2)+s_3=s_1+(s_2+s_3)$. That is, the partial loop $L_B$ is associative, and hence, a partial group. In a similar way, we have that $s_1+s_2=s_2+s_1$, for all $s_1,s_2\in \mathrm{Symb}(B)$. That is, $L_B$ is also abelian. The fact that $L_\mathcal{D}$ is also a partial abelian group follows readily from its definition.
\end{proof}

\vspace{0.25cm}

Under the assumptions, of Lemma \ref{lemma_existence}, the cardinality of the set $\mathrm{Symb}(B)$, with $B\in\mathcal{D}$, plays a relevant role in the study of the partial abelian group $L_B$. Thus, for example, if $\mathrm{Symb}(B)=\{s_1,s_2,s_3\}$, then $s_i+s_j=-s_k$ whenever $\{i,j,k\}=\{1,2,3\}$. This implies that

\begin{equation}\label{eq:LB3}
L_B\equiv\begin{array}{|c||c|c|c|c|c|c|}\hline
0 & s_1 & s_2 & s_3 & -s_3 & -s_2 & -s_1\\ \hline \hline
s_1 & \cdot & -s_3 & -s_2 & \cdot &\cdot & 0 \\ \hline
s_2 & -s_3 & \cdot & -s_1 & \cdot & 0 & \cdot \\ \hline
s_3 & -s_2 & -s_1 & \cdot & 0 &  \cdot & \cdot \\ \hline
-s_3 & \cdot & \cdot & 0 & \cdot & \cdot & \cdot \\ \hline
-s_2 & \cdot & 0 &  \cdot & \cdot & \cdot & \cdot \\ \hline
-s_1 & 0 & \cdot & \cdot & \cdot & \cdot & \cdot\\ \hline
\end{array}
\end{equation}
The next example illustrates that $L_B$ is useful to determine a partial abelian group over which the partially filled array $A$ in Example \ref{example_Affine} is $\mathcal{D}$-Heffter.

\begin{example}\label{example_PHeffter} Under the assumptions and notations of Example \ref{example_Affine}, we have from (\ref{eq:LB3}) that the partial group in $\mathcal{L}([9])$ of smallest size over which the array $A$ is $\mathcal{D}_A$-Heffter is the following. (For a much better understanding, those cells related to the same block in Example \ref{example_Affine} have been colored according to the coloring therein described.)

\begin{table}[H]
\resizebox{\textwidth}{!}{
$\begin{array}{|c||c|c|c|c|c|c|c|c|c|c|c|c|c|c|c|c|c|c|c|}\hline
0 & 1 & 2 & 3 & 4 & 5 & 6 & 7 & 8 & 9 & -9 & -8 & -7 & -6 & -5 & -4 & -3 & -2 & -1 \\ \hline \hline
1 & \cdot & \cellcolor{red!100}{\color{white} -3} &\cellcolor{red!100}{\color{white} -2} & \cellcolor{orange!100}{\color{white}-7} &\cellcolor{black!100}{\color{white}  -9} & \cellcolor{blue!100}{\color{white} -8} & \cellcolor{orange!100}{\color{white}-4} & \cellcolor{blue!100}{\color{white} -6} & \cellcolor{black!100}{\color{white} -5} & \cdot & \cdot & \cdot & \cdot & \cdot & \cdot & \cdot & \cdot & 0\\ \hline
2 & \cellcolor{red!100}{\color{white} -3} & \cdot & \cellcolor{red!100}{\color{white} -1} & \cellcolor{blue!50}{\color{white} -9} & \cellcolor{orange!50}{\color{white}-8} & \cellcolor{black!50}{\color{white} -7} & \cellcolor{black!50}{\color{white} -6} & \cellcolor{orange!50}{\color{white}-5} & \cellcolor{blue!50}{\color{white} -4} & \cdot & \cdot & \cdot & \cdot & \cdot & \cdot & \cdot & 0 & \cdot \\ \hline
3 & \cellcolor{red!100}{\color{white} -2} & \cellcolor{red!100}{\color{white} -1} & \cdot & \cellcolor{black!25}{\color{white} -8} &  \cellcolor{blue!25}{\color{white} -7} & \cellcolor{orange!25}{\color{white}-9} & \cellcolor{blue!25}{\color{white} -5 }& \cellcolor{black!25}{\color{white}-4} & \cellcolor{orange!25}{\color{white}-6} & \cdot & \cdot & \cdot & \cdot & \cdot & \cdot & 0 & \cdot & \cdot \\ \hline
4 & \cellcolor{orange!100}{\color{white}-7} & \cellcolor{blue!50}{\color{white} -9} & \cellcolor{black!25}{\color{white}-8} & \cdot & \cellcolor{red!50}{\color{white} -6} & \cellcolor{red!50}{\color{white} -5} & \cellcolor{orange!100}{\color{white}-1} & \cellcolor{black!25}{\color{white}-3} & \cellcolor{blue!50}{\color{white} -2} & \cdot & \cdot & \cdot & \cdot & \cdot & 0 & \cdot & \cdot & \cdot \\ \hline
5 & \cellcolor{black!100}{\color{white} -9} & \cellcolor{orange!50}{\color{white}-8} &  \cellcolor{blue!25}{\color{white} -7} & \cellcolor{red!50}{\color{white} -6} & \cdot & \cellcolor{red!50}{\color{white} -4} &  \cellcolor{blue!25}{\color{white} -3} & \cellcolor{orange!50}{\color{white}-2} & \cellcolor{black!100}{\color{white} -1} & \cdot & \cdot & \cdot & \cdot & 0 & \cdot & \cdot & \cdot & \cdot \\ \hline
6 & \cellcolor{blue!100}{\color{white} -8} & \cellcolor{black!50}{\color{white} -7} & \cellcolor{orange!25}{\color{white}-9} & \cellcolor{red!50}{\color{white} -5} & \cellcolor{red!50}{\color{white} -4} & \cdot &\cellcolor{black!50}{\color{white}  -2} & \cellcolor{blue!100}{\color{white} -1} & \cellcolor{orange!25}{\color{white}-3} & \cdot & \cdot & \cdot & 0 & \cdot & \cdot & \cdot & \cdot & \cdot \\ \hline
7 & \cellcolor{orange!100}{\color{white}-4} & \cellcolor{black!50}{\color{white} -6} &  \cellcolor{blue!25}{\color{white} -5} & \cellcolor{orange!100}{\color{white}-1} & \cellcolor{blue!25}{\color{white} -3} & \cellcolor{black!50}{\color{white} -2} & \cdot & \cellcolor{red!25}{\color{white} -9} & \cellcolor{red!25}{\color{white} -8} & \cdot & \cdot & 0 & \cdot & \cdot & \cdot & \cdot & \cdot & \cdot \\ \hline
8 & \cellcolor{blue!100}{\color{white} -6} & \cellcolor{orange!50}{\color{white}-5} & \cellcolor{black!25}{\color{white}-4} & \cellcolor{black!25}{\color{white}-3} & \cellcolor{orange!50}{\color{white}-2} & \cellcolor{blue!100}{\color{white} -1} & \cellcolor{red!25}{\color{white}-9} & \cdot & \cellcolor{red!25}{\color{white}-7} & \cdot & 0 & \cdot & \cdot & \cdot & \cdot & \cdot & \cdot & \cdot \\ \hline
9 & \cellcolor{black!100}{\color{white} -5} & \cellcolor{blue!50}{\color{white} -4} & \cellcolor{orange!25}{\color{white}-6} & \cellcolor{blue!50}{\color{white} -2} &\cellcolor{black!100}{\color{white}  -1} & \cellcolor{orange!25}{\color{white}-3} & \cellcolor{red!25}{\color{white}-8} & \cellcolor{red!25}{\color{white}-7} & \cdot & 0 & \cdot & \cdot & \cdot & \cdot & \cdot & \cdot & \cdot & \cdot \\ \hline
-9 & \cdot & \cdot & \cdot & \cdot & \cdot & \cdot & \cdot & \cdot & 0 & \cdot & \cdot & \cdot & \cdot & \cdot & \cdot & \cdot & \cdot & \cdot \\ \hline
-8 & \cdot & \cdot & \cdot & \cdot & \cdot & \cdot & \cdot & 0 & \cdot & \cdot & \cdot & \cdot & \cdot & \cdot & \cdot & \cdot & \cdot & \cdot \\ \hline
-7 & \cdot & \cdot & \cdot & \cdot & \cdot & \cdot & 0 & \cdot & \cdot & \cdot & \cdot & \cdot & \cdot & \cdot & \cdot & \cdot & \cdot & \cdot \\ \hline
-6 & \cdot & \cdot & \cdot & \cdot & \cdot & 0 & \cdot & \cdot & \cdot & \cdot & \cdot & \cdot & \cdot & \cdot & \cdot & \cdot & \cdot & \cdot \\ \hline
-5  & \cdot & \cdot & \cdot & \cdot  & 0& \cdot & \cdot & \cdot & \cdot & \cdot & \cdot & \cdot & \cdot & \cdot & \cdot & \cdot & \cdot & \cdot \\ \hline
-4 & \cdot & \cdot & \cdot & 0 & \cdot & \cdot & \cdot & \cdot & \cdot & \cdot & \cdot & \cdot & \cdot & \cdot & \cdot & \cdot & \cdot & \cdot \\ \hline
-3  & \cdot & \cdot & 0 & \cdot & \cdot & \cdot & \cdot & \cdot & \cdot & \cdot & \cdot & \cdot & \cdot & \cdot & \cdot & \cdot & \cdot & \cdot \\ \hline
-2 & \cdot & 0& \cdot & \cdot & \cdot & \cdot & \cdot & \cdot & \cdot & \cdot & \cdot & \cdot & \cdot & \cdot & \cdot & \cdot & \cdot & \cdot \\ \hline
-1 & 0 & \cdot & \cdot & \cdot & \cdot & \cdot & \cdot & \cdot & \cdot & \cdot & \cdot & \cdot & \cdot & \cdot & \cdot & \cdot & \cdot & \cdot \\ \hline
\end{array}$}
\end{table}
Even more, the array $A$ is $\mathcal{D}_A$-Heffter over any partial loop (not necessarily associative) in $\mathcal{L}([9])$ whose set of entries contains all the highlighted ones. Concerning the largest possible weight of such a partial loop, note that a possible completion of the previous partial group is the following non-associative loop.

\begin{table}[H]
\resizebox{\textwidth}{!}{
$\begin{array}{|c||c|c|c|c|c|c|c|c|c|c|c|c|c|c|c|c|c|c|c|}\hline
0 & 1 & 2 & 3 & 4 & 5 & 6 & 7 & 8 & 9 & -9 & -8 & -7 & -6 & -5 & -4 & -3 & -2 & -1 \\ \hline \hline
1 & -1 & \cellcolor{red!100}{\color{white} -3} &\cellcolor{red!100}{\color{white} -2} & \cellcolor{orange!100}{\color{white}-7} &\cellcolor{black!100}{\color{white}  -9} & \cellcolor{blue!100}{\color{white} -8} & \cellcolor{orange!100}{\color{white}-4} & \cellcolor{blue!100}{\color{white} -6} & \cellcolor{black!100}{\color{white} -5} & 5 & 6 & 4 & 8 & 9 & 7 & 2 & 3 & 0 \\ \hline
2 & \cellcolor{red!100}{\color{white} -3} & -2 & \cellcolor{red!100}{\color{white} -1} & \cellcolor{blue!50}{\color{white} -9} & \cellcolor{orange!50}{\color{white}-8} & \cellcolor{black!50}{\color{white} -7} & \cellcolor{black!50}{\color{white} -6} & \cellcolor{orange!50}{\color{white}-5} & \cellcolor{blue!50}{\color{white} -4} & 4 & 5 & 6 & 7 & 8 & 9 & 1 & 0 & 3 \\ \hline
3 & \cellcolor{red!100}{\color{white} -2} & \cellcolor{red!100}{\color{white} -1} & -3 & \cellcolor{black!25}{\color{white} -8} &  \cellcolor{blue!25}{\color{white} -7} & \cellcolor{orange!25}{\color{white}-9} & \cellcolor{blue!25}{\color{white} -5 }& \cellcolor{black!25}{\color{white}-4} & \cellcolor{orange!25}{\color{white}-6} & 6 & 4 & 5 & 9 & 7 & 8 & 0 & 1 & 2 \\ \hline
4 & \cellcolor{orange!100}{\color{white}-7} & \cellcolor{blue!50}{\color{white} -9} & \cellcolor{black!25}{\color{white}-8} & -4 & \cellcolor{red!50}{\color{white} -6} & \cellcolor{red!50}{\color{white} -5} & \cellcolor{orange!100}{\color{white}-1} & \cellcolor{black!25}{\color{white}-3} & \cellcolor{blue!50}{\color{white} -2} & 2 & 3 & 1 & 5 & 6 & 0 & 8 & 9 & 7 \\ \hline
5 & \cellcolor{black!100}{\color{white} -9} & \cellcolor{orange!50}{\color{white}-8} &  \cellcolor{blue!25}{\color{white} -7} & \cellcolor{red!50}{\color{white} -6} & -5 & \cellcolor{red!50}{\color{white} -4} &  \cellcolor{blue!25}{\color{white} -3} & \cellcolor{orange!50}{\color{white}-2} & \cellcolor{black!100}{\color{white} -1} & 1 & 2 & 3 & 4 & 0 & 6 & 7 & 8 & 9 \\ \hline
6 & \cellcolor{blue!100}{\color{white} -8} & \cellcolor{black!50}{\color{white} -7} & \cellcolor{orange!25}{\color{white}-9} & \cellcolor{red!50}{\color{white} -5} & \cellcolor{red!50}{\color{white} -4} & -6 &\cellcolor{black!50}{\color{white}  -2} & \cellcolor{blue!100}{\color{white} -1} & \cellcolor{orange!25}{\color{white}-3} & 3 & 1 & 2 & 0 & 4 & 5 & 9 & 7 & 8 \\ \hline
7 & \cellcolor{orange!100}{\color{white}-4} & \cellcolor{black!50}{\color{white} -6} &  \cellcolor{blue!25}{\color{white} -5} & \cellcolor{orange!100}{\color{white}-1} & \cellcolor{blue!25}{\color{white} -3} & \cellcolor{black!50}{\color{white} -2} & -7 & \cellcolor{red!25}{\color{white} -9} & \cellcolor{red!25}{\color{white} -8} & 8 & 9 & 0 & 2 & 3 & 1 & 5 & 6 & 4 \\ \hline
8 & \cellcolor{blue!100}{\color{white} -6} & \cellcolor{orange!50}{\color{white}-5} & \cellcolor{black!25}{\color{white}-4} & \cellcolor{black!25}{\color{white}-3} & \cellcolor{orange!50}{\color{white}-2} & \cellcolor{blue!100}{\color{white} -1} & \cellcolor{red!25}{\color{white}-9} & -8 & \cellcolor{red!25}{\color{white}-7} & 7 & 0 & 9 & 1 & 2 & 3 & 4 & 5 & 6 \\ \hline
9 & \cellcolor{black!100}{\color{white} -5} & \cellcolor{blue!50}{\color{white} -4} & \cellcolor{orange!25}{\color{white}-6} & \cellcolor{blue!50}{\color{white} -2} &\cellcolor{black!100}{\color{white}  -1} & \cellcolor{orange!25}{\color{white}-3} & \cellcolor{red!25}{\color{white}-8} & \cellcolor{red!25}{\color{white}-7} & -9 & 0 & 7 & 8 & 3 & 1 & 2 & 6 & 4 & 5 \\ \hline
-9 & 5 & 4 & 6 & 2 & 1 & 3 & 8 & 7 & 0 & 9 & -7 & -8 & -3 & -1 & -2 & -6 & -4 & -5 \\ \hline
-8 & 6 & 5 & 4 & 3 & 2 & 1 & 9 & 0 & 7 & -7 & 8 & -9 & -1 & -2 & -3 & -4 & -5 & -6 \\ \hline
-7 & 4 & 6 & 5 & 1 & 3 & 2 & 0 & 9 & 8 & -8 & -9 & 7 & -2 & -3 & -1 & -5 & -6 & -4 \\ \hline
-6 & 8 & 7 & 9 & 5 & 4 & 0 & 2 & 1 & 3 & -3 & -1 & -2 & 6 & -4 & -5 & -9 & -7 & -8 \\ \hline
-5 & 9 & 8 & 7 & 6 & 0 & 4 & 3 & 2 & 1 & -1 & -2 & -3 & -4 & 5 & -6 & -7 & -8 & -9 \\ \hline
-4 & 7 & 9 & 8 & 0 & 6 & 5 & 1 & 3 & 2 & -2 & -3 & -1 & -5 & -6 & 4 & -8 & -9 & -7 \\ \hline
-3 & 2 & 1 & 0 & 8 & 7 & 9 & 5 & 4 & 6 & -6 & -4 & -5 & -9 & -7 & -8 & 3 & -1 & -2 \\ \hline
-2 & 3 & 0 & 1 & 9 & 8 & 7 & 6 & 5 & 4 & -4 & -5 & -6 & -7 & -8 & -9 & -1 & 2 & -3 \\ \hline
-1 & 0 & 3 & 2 & 7 & 9 & 8 & 4 & 6 & 5 & -5 & -6 & -4 & -8 & -9 & -7 & -2 & -3 & 1 \\ \hline
\end{array}$}
\end{table}
To see that it is not associative, note, for example, that $(1+2)+4=-3+4=8\neq 5=1-9=1+(2+4)$. \hfill $\lhd$
\end{example}

\vspace{0.5cm}

If $|\mathrm{Symb}(B)|>3$, then the distribution of entries in $\mathrm{Ent}(L_B)$ is not uniquely determined. Thus, for instance, if $\mathrm{Symb}(B)=\left\{s_1,s_2,s_3,s_4\right\}$, then we have that $s_i+s_j+s_k=-s_l$ and $s_i+s_j=-(s_k+s_l)$ whenever $\{i,j,k,l\}=\{1,2,3,4\}$. It implies the existence of six symbols $\alpha,\beta,\gamma,\delta,\epsilon,\varphi\in \widetilde{S}$ such that the following partial loop is embedded into $L_B$.

\begin{equation}\label{eq:LB4}
\begin{array}{|c||c|c|c|c|c|c|c|c|}\hline
0 & s_1 & s_2 & s_3 & s_4 & -s_4 &  -s_3 & -s_2 & -s_1\\ \hline \hline
s_1 & \cdot & \alpha & \beta & \gamma & \cdot & \cdot &\cdot & 0 \\ \hline
s_2 & \alpha & \cdot & \delta & \epsilon & \cdot & \cdot &0 & \cdot \\ \hline
s_3 & \beta & \delta & \cdot & \varphi & \cdot & 0 &\cdot & \cdot \\ \hline
s_4 & \gamma & \epsilon & \varphi & \cdot & 0 & \cdot &\cdot & \cdot \\ \hline
-s_4 & \cdot & \cdot & \cdot & 0 & \cdot & \alpha &\beta & \delta \\ \hline
-s_3 & \cdot & \cdot & 0 & \cdot & \alpha & \cdot &\gamma & \epsilon \\ \hline
-s_2 & \cdot & 0 & \cdot & \cdot & \beta & \gamma &\cdot & \varphi \\ \hline
-s_1 & 0 & \cdot & \cdot & \cdot & \delta & \epsilon &\varphi & \cdot \\ \hline
\end{array}
\end{equation}

\vspace{0.2cm}

\begin{example}\label{example_PHeffter_B} The distribution of entries described in (\ref{eq:LB3}) and (\ref{eq:LB4}) enables us to ensure that, under the assumptions of Example \ref{example_Affine_B}, the array $B$ is $\mathcal{D}_B$-Heffter over every partial loop whose Cayley table contains the following partial Latin square. (We make use of the block coloring described in Example \ref{example_Affine_B} for a much better understanding.)

\begin{table}[H]
\resizebox{\textwidth}{!}{
$\begin{array}{|c||c|c|c|c|c|c|c|c|c|c|c|c|c|c|c|c|c|c|c|c|c|c|c|c|c|}\hline
0 & 1 & 2 & 3 & 4 & 5 & 6 & 7 & 8 & 9 & 10 & 11 & 12 & -12 & -11 & -10 & -9 & -8 & -7 & -6 & -5 & -4 & -3 & -2 & -1 \\ \hline \hline
1& \cdot&\cellcolor{red!100}{\color{white} -3}  &\cellcolor{red!100}{\color{white} -2}  &\cellcolor{orange!100}{\color{white} 7} &\cdot &\cdot &\cellcolor{orange!100}{\color{white} 10} &\cdot &\cellcolor{blue!100}{\color{white} -11} &\cellcolor{orange!100}{\color{white} 4} &\cellcolor{blue!100}{\color{white} -9} &\cdot &\cdot &\cdot &\cellcolor{orange!100}{\color{white} -7} &\cdot &\cdot &\cellcolor{orange!100}{\color{white} -4} &\cdot &\cdot &\cellcolor{orange!100}{\color{white} -10} &\cdot &\cdot &0 \\ \hline
2 & \cellcolor{red!100}{\color{white} -3} &\cdot &\cellcolor{red!100}{\color{white} -1}  & \cellcolor{blue!50}{\color{white} -12}&\cellcolor{orange!50}{\color{white} 8} &\cdot &\cdot &\cellcolor{orange!50}{\color{white} 11} &\cdot &\cdot &\cellcolor{orange!50}{\color{white} 5} &\cellcolor{blue!50}{\color{white} -4} &\cdot &\cellcolor{orange!50}{\color{white} -8} &\cdot &\cdot &\cellcolor{orange!50}{\color{white} -5} &\cdot &\cdot &\cellcolor{orange!50}{\color{white} -11} &\cdot &\cdot &0 &\cdot \\ \hline
3& \cellcolor{red!100}{\color{white} -2} &\cellcolor{red!100}{\color{white} -1}  &\cdot &\cdot &\cellcolor{blue!30}{\color{white} -7} &\cellcolor{orange!30}{\color{white} 9} &\cellcolor{blue!30}{\color{white} -5} &\cdot &\cellcolor{orange!30}{\color{white} 12}  &\cdot & \cdot &\cellcolor{orange!30}{\color{white} 6}  &\cellcolor{orange!30}{\color{white} -9}  &\cdot &\cdot &\cellcolor{orange!30}{\color{white} -6}  &\cdot &\cdot&\cellcolor{orange!30}{\color{white} -12}  &\cdot &\cdot &0 &\cdot &\cdot \\ \hline
4& \cellcolor{orange!100}{\color{white} 7}& \cellcolor{blue!50}{\color{white} -12}& \cdot & \cdot & \cellcolor{red!50}{\color{white} -6} & \cellcolor{red!50}{\color{white} -5}& \cellcolor{orange!100}{\color{white} -4}& \cdot & \cdot & \cellcolor{orange!100}{\color{white} -10}& \cdot & \cellcolor{blue!50}{\color{white} -2}& \cdot & \cdot & \cellcolor{orange!100}{\color{white} 1}& \cdot & \cdot & \cellcolor{orange!100}{\color{white} -1}& \cdot & \cdot & 0& \cdot & \cdot & \cellcolor{orange!100}{\color{white} 10}\\ \hline
5 & \cdot & \cellcolor{orange!50}{\color{white} 8}&\cellcolor{blue!30}{\color{white} -7} &\cellcolor{red!50}{\color{white} -6} &\cdot  &\cellcolor{red!50}{\color{white} -4} &\cellcolor{blue!30}{\color{white} -3}  &\cellcolor{orange!50}{\color{white} -5} &\cdot  &\cdot  &\cellcolor{orange!50}{\color{white} -11} &\cdot &\cdot &\cellcolor{orange!50}{\color{white} 2} &\cdot &\cdot &\cellcolor{orange!50}{\color{white} -2} &\cdot &\cdot &0 &\cdot &\cdot &\cellcolor{orange!50}{\color{white} 11} &\cdot \\ \hline
6&\cdot &\cdot &\cellcolor{orange!30}{\color{white} 9}  &\cellcolor{red!50}{\color{white} -5} &\cellcolor{red!50}{\color{white} -4} &\cdot &\cdot &\cellcolor{blue!15}{\color{white} -10}  &\cellcolor{orange!30}{\color{white} -6}  &\cellcolor{blue!15}{\color{white} -8}  &\cdot &\cellcolor{orange!30}{\color{white} -12}  &\cellcolor{orange!30}{\color{white} 3}  &\cdot &\cdot &\cellcolor{orange!30}{\color{white} -3}  &\cdot &\cdot &0 &\cdot & \cdot & \cellcolor{orange!30}{\color{white} 12} & \cdot & \cdot\\ \hline
7 & \cellcolor{orange!100}{\color{white} 10}& \cdot & \cellcolor{blue!30}{\color{white} -5} & \cellcolor{orange!100}{\color{white} -4}& \cellcolor{blue!30}{\color{white} -3} & \cdot & \cdot & \cellcolor{red!30}{\color{white} -9}& \cellcolor{red!30}{\color{white} -8}& \cellcolor{orange!100}{\color{white} -7}& \cdot & \cdot & \cdot & \cdot & \cellcolor{orange!100}{\color{white} -1}& \cdot & \cdot & 0& \cdot & \cdot & \cellcolor{orange!100}{\color{white} 1}& \cdot & \cdot & \cellcolor{orange!100}{\color{white} 4}\\ \hline
8 & \cdot & \cellcolor{orange!50}{\color{white} 11}& \cdot & \cdot & \cellcolor{orange!50}{\color{white} -5}& \cellcolor{blue!15}{\color{white} -10} & \cellcolor{red!30}{\color{white} -9}& \cdot & \cellcolor{red!30}{\color{white} -7}& \cellcolor{blue!15}{\color{white} -6}& \cellcolor{blue!15}\cellcolor{orange!50}{\color{white} -8} & \cdot & \cdot & \cellcolor{orange!50}{\color{white} -2}& \cdot & \cdot & 0& \cdot & \cdot & \cellcolor{orange!50}{\color{white} 2}& \cdot & \cdot & \cellcolor{orange!50}{\color{white} 5}& \cdot \\ \hline
9 & \cellcolor{blue!100}{\color{white} -11}& \cdot & \cellcolor{orange!30}{\color{white} 12} & \cdot & \cdot & \cellcolor{orange!30}{\color{white} -6} & \cellcolor{red!30}{\color{white} -8}& \cellcolor{red!30}{\color{white} -7}& \cdot & \cdot & \cellcolor{blue!100}{\color{white} -1}& \cellcolor{orange!30}{\color{white} -9} & \cellcolor{orange!30}{\color{white} -3} & \cdot & \cdot & 0& \cdot & \cdot & \cellcolor{orange!30}{\color{white} 3} & \cdot & \cdot & \cellcolor{orange!30}{\color{white} 6}  & \cdot & \cdot \\ \hline
10 & \cellcolor{orange!100}{\color{white} 4}& \cdot & \cdot & \cellcolor{orange!100}{\color{white} -10}& \cdot & \cellcolor{blue!15}{\color{white} -8} & \cellcolor{orange!100}{\color{white} -7}& \cellcolor{blue!15}{\color{white} -6} & \cdot & \cdot & \cellcolor{red!15}{\color{white} -12}& \cellcolor{red!15}{\color{white} -11}& \cdot & \cdot & 0& \cdot & \cdot & \cellcolor{orange!100}{\color{white} 1}& \cdot & \cdot & \cellcolor{orange!100}{\color{white} -1}& \cdot & \cdot & \cellcolor{orange!100}{\color{white} 7}\\ \hline
11&\cellcolor{blue!100}{\color{white} -9} &\cellcolor{orange!50}{\color{white} 5} &\cdot &\cdot &\cellcolor{orange!50}{\color{white} -11} &\cdot &\cdot &\cellcolor{orange!50}{\color{white} -8}  &\cellcolor{blue!100}{\color{white} -1} &\cellcolor{red!15}{\color{white} -12} &\cdot &\cellcolor{red!15}{\color{white} -10} &\cdot &0 &\cdot & \cdot&\cellcolor{orange!50}{\color{white} 2} &\cdot &\cdot &\cellcolor{orange!50}{\color{white} -2} &\cdot &\cdot &\cellcolor{orange!50}{\color{white} 8} &\cdot \\ \hline
12& \cdot& \cellcolor{blue!50}{\color{white} -4}& \cellcolor{orange!30}{\color{white} 6} & \cellcolor{blue!50}{\color{white} -2}& \cdot& \cellcolor{orange!30}{\color{white} -12} & \cdot& \cdot& \cellcolor{orange!30}{\color{white} -9} & \cellcolor{red!15}{\color{white} -11}& \cellcolor{red!15}{\color{white} -10}& \cdot & 0& \cdot & \cdot & \cellcolor{orange!30}{\color{white} 3} & \cdot & \cdot & \cellcolor{orange!30}{\color{white} -3} & \cdot & \cdot & \cellcolor{orange!30}{\color{white} 9} & \cdot & \cdot \\ \hline
-12 & \cdot & \cdot & \cellcolor{orange!30}{\color{white} -9} & \cdot & \cdot & \cellcolor{orange!30}{\color{white} 3} & \cdot & \cdot & \cellcolor{orange!30}{\color{white} -3} & \cdot & \cdot & 0& \cdot & \cdot & \cdot & \cellcolor{orange!30}{\color{white} 9} & \cdot & \cdot & \cellcolor{orange!30}{\color{white} 12} & \cdot & \cdot & -\cellcolor{orange!30}{\color{white} -6} & \cdot & \cdot \\ \hline
-11 & \cdot & \cellcolor{orange!50}{\color{white} -8}& \cdot & \cdot & \cellcolor{orange!50}{\color{white} 2}& \cdot & \cdot & \cellcolor{orange!50}{\color{white} -2}& \cdot & \cdot & 0& \cdot & \cdot & \cdot & \cdot & \cdot & \cellcolor{orange!50}{\color{white} 8}& \cdot & \cdot & \cellcolor{orange!50}{\color{white} 11}& \cdot & \cdot & \cellcolor{orange!50}{\color{white} -5}& \cdot \\ \hline
-10 & \cellcolor{orange!100}{\color{white} -7}& \cdot & \cdot & \cellcolor{orange!100}{\color{white} 1}& \cdot & \cdot & \cellcolor{orange!100}{\color{white} -1}& \cdot & \cdot & 0& \cdot & \cdot & \cdot & \cdot & \cdot & \cdot & \cdot & \cellcolor{orange!100}{\color{white} 7}& \cdot & \cdot & \cellcolor{orange!100}{\color{white} 10}& \cdot & \cdot & \cellcolor{orange!100}{\color{white} -4}\\ \hline
-9 & \cdot & \cdot & \cellcolor{orange!30}{\color{white} -6} & \cdot & \cdot & \cellcolor{orange!30}{\color{white} -3} & \cdot & \cdot & 0& \cdot & \cdot & \cellcolor{orange!30}{\color{white} 3} & \cellcolor{orange!30}{\color{white} 9} & \cdot & \cdot &\cdot & \cdot & \cdot & \cellcolor{orange!30}{\color{white} 6} & \cdot & \cdot & \cellcolor{orange!30}{\color{white} -12} & \cdot & \cdot \\ \hline
-8 & \cdot & \cellcolor{orange!50}{\color{white} -5}& \cdot & \cdot & \cellcolor{orange!50}{\color{white} -2}& \cdot & \cdot & 0& \cdot & \cdot & \cellcolor{orange!50}{\color{white} 2}& \cdot & \cdot & \cellcolor{orange!50}{\color{white} 8}& \cdot & \cdot & \cdot & \cdot & \cdot & \cellcolor{orange!50}{\color{white} 5}& \cdot & \cdot & \cellcolor{orange!50}{\color{white} -11}& \cdot\\ \hline
-7& \cellcolor{orange!100}{\color{white} -4} & \cdot & \cdot & \cellcolor{orange!100}{\color{white} -1}& \cdot & \cdot & 0& \cdot & \cdot & \cellcolor{orange!100}{\color{white} 1}& \cdot & \cdot & \cdot & \cdot & \cellcolor{orange!100}{\color{white} 7}& \cdot & \cdot & \cdot & \cdot & \cdot & \cellcolor{orange!100}{\color{white} 4}& \cdot & \cdot & \cellcolor{orange!100}{\color{white} -10}\\ \hline
-6& \cdot & \cdot & \cellcolor{orange!30}{\color{white} -12} & \cdot & \cdot & 0& \cdot & \cdot & \cellcolor{orange!30}{\color{white} 3} & \cdot & \cdot & \cellcolor{orange!30}{\color{white} -3} & \cellcolor{orange!30}{\color{white} 12} & \cdot & \cdot & \cellcolor{orange!30}{\color{white} 6} & \cdot & \cdot & \cdot & \cdot& \cdot & \cellcolor{orange!30}{\color{white} -9} & \cdot & \cdot\\ \hline
-5& \cdot & \cellcolor{orange!50}{\color{white} -11}& \cdot & \cdot  0& \cdot & \cdot & \cellcolor{orange!50}{\color{white} 2}& \cdot & \cdot & \cellcolor{orange!50}{\color{white} -2}& \cdot & \cdot & \cellcolor{orange!50}{\color{white} 11}& \cdot &\cdot  & \cellcolor{orange!50}{\color{white} 5}& \cdot & \cdot & \cdot & \cdot & \cdot & \cellcolor{orange!50}{\color{white} -8} & \cdot\\ \hline
-4& \cellcolor{orange!100}{\color{white} -10}& \cdot & \cdot & 0& \cdot & \cdot & \cellcolor{orange!100}{\color{white} 1}& \cdot & \cdot & \cellcolor{orange!100}{\color{white} -1}& \cdot & \cdot & \cdot & \cdot & \cellcolor{orange!100}{\color{white} 10}& \cdot & \cdot & \cellcolor{orange!100}{\color{white} 4}& \cdot & \cdot & \cdot & \cdot & \cdot & \cellcolor{orange!100}{\color{white} -7}\\ \hline
-3& \cdot & \cdot& 0& \cdot & \cdot & \cellcolor{orange!30}{\color{white} 12} & \cdot & \cdot & \cellcolor{orange!30}{\color{white} 6} & \cdot & \cdot & \cellcolor{orange!30}{\color{white} 9} & \cellcolor{orange!30}{\color{white} -6} & \cdot & \cdot & \cellcolor{orange!30}{\color{white} -12} & \cdot & \cdot & \cellcolor{orange!30}{\color{white} -9} & \cdot & \cdot & \cdot & \cdot & \cdot \\ \hline
-2 & \cdot & 0& \cdot & \cdot & \cellcolor{orange!50}{\color{white} 11}& \cdot & \cdot & \cellcolor{orange!50}{\color{white} 5}& \cdot & \cdot & \cellcolor{orange!50}{\color{white} 8}& \cdot & \cdot & \cellcolor{orange!50}{\color{white} -5}& \cdot & \cdot & \cellcolor{orange!50}{\color{white} -11}& \cdot & \cdot & \cellcolor{orange!50}{\color{white} -8}& \cdot & \cdot & \cdot & \cdot \\ \hline
-1& 0& \cdot & \cdot & \cellcolor{orange!100}{\color{white} 10}& \cdot & \cdot & \cellcolor{orange!100}{\color{white} 4}& \cdot & \cdot & \cellcolor{orange!100}{\color{white} 7}& \cdot & \cdot & \cdot & \cdot & \cellcolor{orange!100}{\color{white} -4}& \cdot & \cdot & \cellcolor{orange!100}{\color{white} -10}& \cdot & \cdot & \cellcolor{orange!100}{\color{white} -7}& \cdot & \cdot & \cdot\\ \hline
\end{array}$}
\end{table}

A possible completion of this partial Latin square is the following loop.

\begin{table}[H]
\resizebox{\textwidth}{!}{
$\begin{array}{|c||c|c|c|c|c|c|c|c|c|c|c|c|c|c|c|c|c|c|c|c|c|c|c|c|c|}\hline
0 & 1 & 2 & 3 & 4 & 5 & 6 & 7 & 8 & 9 & 10 & 11 & 12 & -12 & -11 & -10 & -9 & -8 & -7 & -6 & -5 & -4 & -3 & -2 & -1 \\ \hline \hline
1& -1&\cellcolor{red!100}{\color{white} -3}  &\cellcolor{red!100}{\color{white} -2}  &\cellcolor{orange!100}{\color{white} 7} &6 &2 &\cellcolor{orange!100}{\color{white} 10} &3 &\cellcolor{blue!100}{\color{white} -11} &\cellcolor{orange!100}{\color{white} 4} &\cellcolor{blue!100}{\color{white} -9} &5 &8 &-12 &\cellcolor{orange!100}{\color{white} -7} &-8 &12 &\cellcolor{orange!100}{\color{white} -4} &11 &9 &\cellcolor{orange!100}{\color{white} -10} &-5 &-6 &0 \\ \hline
2& \cellcolor{red!100}{\color{white} -3} &-2 &\cellcolor{red!100}{\color{white} -1}  & \cellcolor{blue!50}{\color{white} -12}&\cellcolor{orange!50}{\color{white} 8} &1 &3 &\cellcolor{orange!50}{\color{white} 11} &4 &9 &\cellcolor{orange!50}{\color{white} 5} &\cellcolor{blue!50}{\color{white} -4} &7 &\cellcolor{orange!50}{\color{white} -8} &-6 &10 &\cellcolor{orange!50}{\color{white} -5} &12 &-10 &\cellcolor{orange!50}{\color{white} -11} &-9 &-7 &0 &6 \\ \hline
3& \cellcolor{red!100}{\color{white} -2} &\cellcolor{red!100}{\color{white} -1}  &-3 &2 &\cellcolor{blue!30}{\color{white} -7} &\cellcolor{orange!30}{\color{white} 9} &\cellcolor{blue!30}{\color{white} -5} &1 &\cellcolor{orange!30}{\color{white} 12}  &8 & 4 &\cellcolor{orange!30}{\color{white} 6}  &\cellcolor{orange!30}{\color{white} -9}  &-10 &-11 &\cellcolor{orange!30}{\color{white} -6}  &-4 &11 &\cellcolor{orange!30}{\color{white} -12}  &10 &-8 &0 &7 &5 \\ \hline
4& \cellcolor{orange!100}{\color{white} 7}& \cellcolor{blue!50}{\color{white} -12}& 2& -7& \cellcolor{red!50}{\color{white} -6} & \cellcolor{red!50}{\color{white} -5}& \cellcolor{orange!100}{\color{white} -4}& 12& 11& \cellcolor{orange!100}{\color{white} -10}& 6& \cellcolor{blue!50}{\color{white} -2}& -11& 3& \cellcolor{orange!100}{\color{white} 1}& 5& -9& \cellcolor{orange!100}{\color{white} -1}& -8& -3& 0& 8& 9& \cellcolor{orange!100}{\color{white} 10}\\ \hline
5& 6& \cellcolor{orange!50}{\color{white} 8}&\cellcolor{blue!30}{\color{white} -7} &\cellcolor{red!50}{\color{white} -6} &-8 &\cellcolor{red!50}{\color{white} -4} &\cellcolor{blue!30}{\color{white} -3}  &\cellcolor{orange!50}{\color{white} -5} &10 &12 &\cellcolor{orange!50}{\color{white} -11} &7 &-1 &\cellcolor{orange!50}{\color{white} 2} &9 &4 &\cellcolor{orange!50}{\color{white} -2} &-12 &1 &0 &3 &-10 &\cellcolor{orange!50}{\color{white} 11} &-9 \\ \hline
6&2 &1 &\cellcolor{orange!30}{\color{white} 9}  &\cellcolor{red!50}{\color{white} -5} &\cellcolor{red!50}{\color{white} -4} &-9 &11 &\cellcolor{blue!15}{\color{white} -10}  &\cellcolor{orange!30}{\color{white} -6}  &\cellcolor{blue!15}{\color{white} -8}  &7 &\cellcolor{orange!30}{\color{white} -12}  &\cellcolor{orange!30}{\color{white} 3}  &4 &5 &\cellcolor{orange!30}{\color{white} -3}  &-7 &-2 &0 &-1 & 8& \cellcolor{orange!30}{\color{white} 12} & 10& -11\\ \hline
7& \cellcolor{orange!100}{\color{white} 10}& 3& \cellcolor{blue!30}{\color{white} -5} & \cellcolor{orange!100}{\color{white} -4}& \cellcolor{blue!30}{\color{white} -3} & 11& -10& \cellcolor{red!30}{\color{white} -9}& \cellcolor{red!30}{\color{white} -8}& \cellcolor{orange!100}{\color{white} -7}& 9& 8& 5& 6& \cellcolor{orange!100}{\color{white} -1}& -2& -6& 0& 2& 12& \cellcolor{orange!100}{\color{white} 1}& -11& -12& \cellcolor{orange!100}{\color{white} 4}\\ \hline
8 & 3& \cellcolor{orange!50}{\color{white} 11}& 1& 12& \cellcolor{orange!50}{\color{white} -5}& \cellcolor{blue!15}{\color{white} -10} & \cellcolor{red!30}{\color{white} -9}& -11& \cellcolor{red!30}{\color{white} -7}& \cellcolor{blue!15}{\color{white} -6}& \cellcolor{blue!15}\cellcolor{orange!50}{\color{white} -8} & 10& -4& \cellcolor{orange!50}{\color{white} -2}& -3& -1& 0& 6& 7& \cellcolor{orange!50}{\color{white} 2}& 9& 4& \cellcolor{orange!50}{\color{white} 5}& -12\\ \hline
9 & \cellcolor{blue!100}{\color{white} -11}& 4& \cellcolor{orange!30}{\color{white} 12} & 11& 10& \cellcolor{orange!30}{\color{white} -6} & \cellcolor{red!30}{\color{white} -8}& \cellcolor{red!30}{\color{white} -7}& -12& 5& \cellcolor{blue!100}{\color{white} -1}& \cellcolor{orange!30}{\color{white} -9} & \cellcolor{orange!30}{\color{white} -3} & 7& -2& 0& 1& 2& \cellcolor{orange!30}{\color{white} 3} & -4& -5& \cellcolor{orange!30}{\color{white} 6}  & -10& 8\\ \hline
10 & \cellcolor{orange!100}{\color{white} 4}& 9& 8& \cellcolor{orange!100}{\color{white} -10}& 12& \cellcolor{blue!15}{\color{white} -8} & \cellcolor{orange!100}{\color{white} -7}& \cellcolor{blue!15}{\color{white} -6} & 5& -4& \cellcolor{red!15}{\color{white} -12}& \cellcolor{red!15}{\color{white} -11}& -2& -3& 0& 2& 3& \cellcolor{orange!100}{\color{white} 1}& -5& -9& \cellcolor{orange!100}{\color{white} -1}& 11& 6& \cellcolor{orange!100}{\color{white} 7}\\ \hline
11&\cellcolor{blue!100}{\color{white} -9} &\cellcolor{orange!50}{\color{white} 5} &4 &6 &\cellcolor{orange!50}{\color{white} -11} &7 &9 &\cellcolor{orange!50}{\color{white} -8}  &\cellcolor{blue!100}{\color{white} -1} &\cellcolor{red!15}{\color{white} -12} &-5 &\cellcolor{red!15}{\color{white} -10} &1 &0 &3 & -7&\cellcolor{orange!50}{\color{white} 2} &-6 &-4 &\cellcolor{orange!50}{\color{white} -2} &-3 &10 &\cellcolor{orange!50}{\color{white} 8} &12 \\ \hline
12& 5& \cellcolor{blue!50}{\color{white} -4}& \cellcolor{orange!30}{\color{white} 6} & \cellcolor{blue!50}{\color{white} -2}& 7& \cellcolor{orange!30}{\color{white} -12} & 8& 10& \cellcolor{orange!30}{\color{white} -9} & \cellcolor{red!15}{\color{white} -11}& \cellcolor{red!15}{\color{white} -10}& -6& 0& -1& 2& \cellcolor{orange!30}{\color{white} 3} & 4& -5& \cellcolor{orange!30}{\color{white} -3} & 1& 11& \cellcolor{orange!30}{\color{white} 9} & -7& -8\\ \hline
-12 & 8& 7& \cellcolor{orange!30}{\color{white} -9} & -11& -1& \cellcolor{orange!30}{\color{white} 3} & 5& -4& \cellcolor{orange!30}{\color{white} -3} & -2& 1 & 0& 6& 10& 11& \cellcolor{orange!30}{\color{white} 9} & -10& -8& \cellcolor{orange!30}{\color{white} 12} & -7& 2& -\cellcolor{orange!30}{\color{white} -6} & 4& -5\\ \hline
-11 & -12& \cellcolor{orange!50}{\color{white} -8}& -10& 3& \cellcolor{orange!50}{\color{white} 2}& 4& 6& \cellcolor{orange!50}{\color{white} -2}& 7& -3& 0& -1& 10& 5& 12& 1& \cellcolor{orange!50}{\color{white} 8}& -9& -7& \cellcolor{orange!50}{\color{white} 11}& -6& -4& \cellcolor{orange!50}{\color{white} -5}& 9\\ \hline
-10 & \cellcolor{orange!100}{\color{white} -7}& -6& -11& \cellcolor{orange!100}{\color{white} 1}& 9& 5& \cellcolor{orange!100}{\color{white} -1}& -3& -2& 0& 3& 2& 11& 12& 4& -5& 6& \cellcolor{orange!100}{\color{white} 7}& 8& -12& \cellcolor{orange!100}{\color{white} 10}& -8& -9& \cellcolor{orange!100}{\color{white} -4}\\ \hline
-9& -8& 10& \cellcolor{orange!30}{\color{white} -6} & 5& 4& \cellcolor{orange!30}{\color{white} -3} & -2& -1& 0& 2& -7& \cellcolor{orange!30}{\color{white} 3} & \cellcolor{orange!30}{\color{white} 9} & 1& -5& 12& 7& 8& \cellcolor{orange!30}{\color{white} 6} & -10& -11& \cellcolor{orange!30}{\color{white} -12} & -4& 11\\ \hline
-8 & 12& \cellcolor{orange!50}{\color{white} -5}& -4& -9& \cellcolor{orange!50}{\color{white} -2}& -7& -6& 0& 1& 3& \cellcolor{orange!50}{\color{white} 2}& 4& -10& \cellcolor{orange!50}{\color{white} 8}& 6& 7& 11& 9& 10& \cellcolor{orange!50}{\color{white} 5}& -12& -1& \cellcolor{orange!50}{\color{white} -11}& -3\\ \hline
-7& \cellcolor{orange!100}{\color{white} -4} & 12& 11& \cellcolor{orange!100}{\color{white} -1}& -12& -2& 0& 6& 2& \cellcolor{orange!100}{\color{white} 1}& -6& -5& -8& -9& \cellcolor{orange!100}{\color{white} 7}& 8& 9& 10& -11& 3& \cellcolor{orange!100}{\color{white} 4}& 5& -3& \cellcolor{orange!100}{\color{white} -10}\\ \hline
-6& 11& -10& \cellcolor{orange!30}{\color{white} -12} & -8& 1& 0& 2& 7& \cellcolor{orange!30}{\color{white} 3} & -5& -4& \cellcolor{orange!30}{\color{white} -3} & \cellcolor{orange!30}{\color{white} 12} & -7& 8& \cellcolor{orange!30}{\color{white} 6} & 10& -11& 9 & 4& 5& \cellcolor{orange!30}{\color{white} -9} & -1& -2\\ \hline
-5& 9& \cellcolor{orange!50}{\color{white} -11}& 10& -3& 0& -1& 12& \cellcolor{orange!50}{\color{white} 2}& -4& -9& \cellcolor{orange!50}{\color{white} -2}& 1& -7& \cellcolor{orange!50}{\color{white} 11}& -12&-10 & \cellcolor{orange!50}{\color{white} 5}& 3& 4& 8& 6& 7& \cellcolor{orange!50}{\color{white} -8} & -6\\ \hline
-4& \cellcolor{orange!100}{\color{white} -10}& -9& -8& 0& 3& 8& \cellcolor{orange!100}{\color{white} 1}& 9& -5& \cellcolor{orange!100}{\color{white} -1}& -3& 11& 2& -6& \cellcolor{orange!100}{\color{white} 10}& -11& -12& \cellcolor{orange!100}{\color{white} 4}& 5& 6& 7& -2& 12& \cellcolor{orange!100}{\color{white} -7}\\ \hline
-3& -5& -7& 0& 8& -10& \cellcolor{orange!30}{\color{white} 12} & -11& 4& \cellcolor{orange!30}{\color{white} 6} & 11& 10& \cellcolor{orange!30}{\color{white} 9} & \cellcolor{orange!30}{\color{white} -6} & -4& -8& \cellcolor{orange!30}{\color{white} -12} & -1& 5& \cellcolor{orange!30}{\color{white} -9} & 7& -2& 3& 1& 2\\ \hline
-2& -6& 0& 7& 9& \cellcolor{orange!50}{\color{white} 11}& 10& -12& \cellcolor{orange!50}{\color{white} 5}& -10& 6& \cellcolor{orange!50}{\color{white} 8}& -7& 4& \cellcolor{orange!50}{\color{white} -5}& -9& -4& \cellcolor{orange!50}{\color{white} -11}& -3& -1& \cellcolor{orange!50}{\color{white} -8}& 12& 1& 2& 3\\ \hline
-1& 0& 6& 5& \cellcolor{orange!100}{\color{white} 10}& -9& -11& \cellcolor{orange!100}{\color{white} 4}& -12& 8& \cellcolor{orange!100}{\color{white} 7}& 12& -8& -5& 9& \cellcolor{orange!100}{\color{white} -4}& 11& -3& \cellcolor{orange!100}{\color{white} -10}& -2& -6& \cellcolor{orange!100}{\color{white} -7}& 2& 3& 1\\ \hline
\end{array}$}
\end{table}

This loop is not associative. Thus, for instance, $(2+3)+4=-1+4=10\neq -2=2+2=2+(3+4)$. \hfill $\lhd$
\end{example}

\section{Conclusion and further works}

This paper has dealt with the existence of Heffter arrays in which the summation of symbols is not computed via an abelian group, but a partial loop. In this regard, we have extended the original notion of Heffter array, not only to the non-commutative case (as it has recently been suggested in \cite{Buratti2023}, and dealt with in \cite{MT}), but also to the non-associative case. Even more, we have delved into the possible use of partial binary operations. To this end, we have introduced the concepts of $\mathcal{P}$- and $\mathcal{D}$-Heffter arrays over an additive partial loop, where $\mathcal{P}$ is a set of all-one polynomials without constant terms, whose variables are related to the blocks of an affine $1$-design $\mathcal{D}$ on the set of entries in the Heffter array.

We have proved the existence of these new objects beyond the case of dealing with abelian groups. This has been illustrated with a construction of $\mathcal{P}$- and $\mathcal{D}$-Heffter arrays over partial loops in which the affine design $\mathcal{D}$ contains many blocks, and every block is associated to a large number of sum polynomials in $\mathcal{P}$. By considering an affine design over the elements of a partially filled array, our construction deals with the existence problem of Heffter linear spaces from a different perspective.  Indeed, we show that it is possible to construct these structures over non-associative loops (which in some cases are also abelian), having many zero sum block polynomials. The construction of subsquares relative to the symbol set of each block shows some necessary conditions that have to be satisfied by a partial loop in order to build a loop that is $\mathcal{D}$-Heffter. Can these conditions show a possible completion of a partial loop to a group, or disprove it? A solution to this problem could respectively lead to existence or non-existence results of Heffter linear spaces over groups.

\section*{Acknowledgements}

The first author is partially supported by both the Research Project FQM-016 {\em ``Codes, Design, Cryptography and Optimization''}, from Junta de Andaluc\'\i a, and the Research Project TED2021-130566B-100 from Ministry of Science and Innovation of the Government of Spain. The second author is  partially supported by INdAM - GNSAGA. The research was also supported by the Italian Ministry of University and Research through the project named “Fondi per attivit\`a a carattere internazionale” 2022 INTER DICATAM PASOTTI.

\end{document}